\documentclass[reqno,12pt]{amsart}
\usepackage{amssymb,amsmath,amsthm}
\usepackage{cite}
\usepackage{mathdots}
\usepackage{dsfont}
\usepackage{tikz}
\usepackage{float}
\usepackage{bm}
\usepackage{graphicx}
\usepackage{subcaption}
\usepackage{caption}
\usepackage{mathabx}

\def\gb{{\bf g}}
\def\hb{{\bf h}}
\def\nb{{\bf n}}

\def\rb{{\bf r}}
\def\ub{{\bf u}}
\def\vb{{\bf v}}

\def\xb{{\bf x}}

\def\Ab{{\bf A}}
\def\Bb{{\bf B}}
\def\Cb{{\bf C}}
\def\Db{{\bf D}}
\def\Hb{{\bf H}}
\def\Ib{{\bf I}}
\def\Jb{{\bf J}}

\def\Pb{{\bf P}}
\def\Rb{{\bf R}}
\def\Sb{{\bf S}}
\def\Tb{{\bf T}}

\def\epsilonb{{\bm \epsilon}}
\def\etab{{\bm \eta}}
\def\gammab{{\bm \gamma}}
\def\omegab{{\bm \omega}}
\def\phib{{\bm \phi}}
\def\thetab{{\bm \theta}}

\newcommand{\0}{{\bf 0}}
\newcommand{\1}{{\bf 1}}

\def\nt{{\tilde{n}}}
\def\st{{\tilde{s}}}

\def\hbt{{\widetilde{\bf h}}}
\def\xbt{{\widetilde{\bf x}}}

\def\Abt{{\widetilde{\bf A}}}
\def\Bbt{{\widetilde{\bf B}}}
\def\Cbt{{\widetilde{\bf C}}}
\def\Dbt{{\widetilde{\bf D}}}
\def\Hbt{{\widetilde{\bf H}}}
\def\Rbt{{\widetilde{\bf R}}}
\def\Sbt{{\widetilde{\bf S}}}
\def\Tbt{{\widetilde{\bf T}}}

\def\etabt{{\widetilde{\bm \eta}}}
\def\thetabt{{\widetilde{\bm \theta}}}

\def\omegao{{\overline{\omega}}}

\def\Ic{{\mathcal{I}}}

\def\red{\text{Red}}

\DeclareMathOperator{\adj}{adj}
\DeclareMathOperator{\sym}{sym}

\def\tb{{\bm \theta}}

\def\be{{\bm \epsilon}}
\def\bg{{\bm \gamma}}
\def\bn{{\bm \eta}}
\newcommand{\bp}{{\bm \phi}}
\newcommand{\bt}{{\bm \theta}}
\newcommand{\bw}{{\bm \omega}}

\newcommand{\g}{{\bf g}}
\def\h{{\bf h}}
\def\u{{\bf u}}
\def\v{{\bf v}}
\newcommand{\x}{{\bf x}}
\newcommand{\y}{{\bf y}}

\def\A{{\bf A}}
\def\B{{\bf B}}
\def\C{{\bf C}}
\def\D{{\bf D}}
\newcommand{\I}{{\bf I}}
\newcommand{\J}{{\bf J}}
\newcommand{\M}{{\bf M}}
\def\P{{\bf P}}

\def\S{{\bf S}}

\def\R{\mathbb{R}}

\newtheorem{theorem}{Theorem}
\newtheorem{lemma}[theorem]{Lemma}

\theoremstyle{remark}

\newtheorem*{example}{Example}

\numberwithin{equation}{section}

\newtheorem{definition}{Definition}

\begin{document}


\title{Subcritical Bifurcations in the Kuramoto Model with Positive Coupling}
\author{Timothy Ferguson$^1$}

\begin{abstract}
The Kuramoto model is a standard model for the dynamics of coupled oscillator networks. In particular, it is used to study long time behavior such as phase-locking where all oscillators rotate at a common frequency with fixed angle differences. It has been observed that phase-locking occurs if the natural frequencies of the oscillators are sufficiently close to each other, and doesn't if they aren't. Intuitively, the closer the natural frequencies are to each other the more stable the system is. In this paper, we study the Kuramoto model on ring networks with positive coupling between the oscillators and derive a criterion for a bifurcation in which two branches of phase-locked solutions collide as we increase the coupling strengths. Furthermore, we state a criterion for when one of these branches consists of stable phase-locked solutions. In this case stability is lost as the positive coupling is increased. We then apply our criteria to show that for any size of the ring network there always exists choices of the natural frequencies and coupling strengths so that this bifurcation occurs. (We require at least five oscillators for one branch to consist of stable phase-locked solutions.) Finally, we conjecture that generically our bifurcation is locally a subcritical bifurcation and globally an $S$-curve. (In the case where one branch consists of phase-locked solutions, an $S$-curve results in bistability i.e. the existence of two distinct stable phase-locked solutions.) Finally, we note that our methods are constructive.
\end{abstract}

\maketitle

\smallskip
\noindent \footnotesize\textbf{Keywords.} Kuramoto model, fixed points, bifurcations, stability
\\
\\
\smallskip
\noindent \footnotesize\textbf{AMS subject classifications.} 34C15, 34C23, 34D20

\footnotetext[1]{Department of Mathematics, Arizona State University, 901 S. Palm Walk, Tempe, AZ 85281, USA}


\section{Introduction}

The Kuramoto model is a model for the dynamics of oscillator networks, or more generally, networks of involving periodic components. In particular, it is used to study synchronization whereby the individual oscillators act in unison and rotate at a common frequency rather than rotate at their initial individual frequencies. The Kuramoto model was first defined by Y. Kuramoto in 1984 in \cite{MR762432} and has many practical applications to periodic phenomena. Examples from physics and engineering include generators in the power grid \cite{Filatrella2008} and laser arrays \cite{PhysRevA.52.4089} while examples from biology include pacemaker cells in the heart \cite{peskin75} and flashing colonies of fireflies \cite{Buck1988SynchronousRhythmicFlashing}. For a more complete discussion of synchronization as well as a history of the development of the Kuramoto model see \cite{Sync.book,MR1783382,Acebron.etal.05}. The Kuramoto model on a ring network with $n$ oscillators is the system of $n$ coupled non-linear differential equations
\begin{align} \label{eq:model}
\frac{d\theta_i}{dt} = \omega_i + \sigma (\gamma_{i+1,i} \sin(\theta_{i+1} - \theta_i) + \gamma_{i-1,i} \sin(\theta_{i-1} - \theta_i)) \quad \text{for} \quad i \in \{1,\dots,n\}
\end{align}
where $\theta_i$ denotes the phase angle of the $i$th oscillator, $\omega_i$ the natural frequency of the $i$th oscillator, and $\gamma_{i-1,i} = \gamma_{i,i-1} > 0$ represents strength of the coupling between $\theta_{i-1}$ and $\theta_i$. (Throughout we will always interpret indices modulo $n$.) Note that the right hand side of \eqref{eq:model} naturally breaks up into two contributions. The first contribution is the natural frequency $\omega_i$ which determines the behavior of $\theta_i$ without any coupling. The second contribution is the term $\gamma_{i+1,i} \sin(\theta_{i+1} - \theta_i) + \gamma_{i-1,i} \sin(\theta_{i-1} - \theta_i)$ which determines the behavior of $\theta_i$ due to its interaction with $\theta_{i-1}$ and $\theta_{i+1}$. The parameter $\sigma > 0$ allows us to change the relative strengths of these two contributions. It is well-known that if $\sigma$ is sufficiently large, then the Kuramoto model \eqref{eq:model} has a unique stable fixed point to which all solutions converge. In this case all oscillators rotate at a common rate and we say that the system synchronizes. For results concerning the critical value of $\sigma$ for the onset of synchronization as well as estimates for specific network structures see \cite{Verwoerd.Mason.2008,Dorfler.Chertkov.Bullo.2013,Dorfler.Bullo.2011,Chopra.Spong.2009,Verwoerd.Mason.2009}. For smaller values of $\sigma$ clusters of the network may synchronize to different frequencies. See \cite{Brady.2018} for diagrams illustrating this phenomenon. In particular, these diagrams illustrate bifurcations that occur when two or more synchronized clusters become a single larger synchronized cluster.

It is natural to ask what kind of bifurcations can occur in the Kuramoto model as we change the parameter $\sigma$. In particular, we will focus on bifurcations involving phase-locked solutions i.e. solutions for which all oscillators rotate at a common frequency with fixed angle differences. In Figure \ref{fig:1a} we give an example plot of all phase-locked solutions as we vary $\sigma$. In order to visualize these plots we represent a phase-locked solution by its order parameter $r = \frac{1}{n} \bigg\rvert \sum_{i=1}^n e^{\theta_i \sqrt{-1}}\bigg\rvert$ which is a measure of the solutions level of synchronization. One might expect that increasing $\sigma$ will only create rather than destroy stability of phase-locked solutions. However, this is not always the case. For example, Terada, Ito, Aoyagi, and Yamaguchi \cite{Terada_2017} numerically discovered a subcritical saddle-node bifurcation whereby a stable phase-locked solution collides with a 1-saddle. They did this by choosing the natural frequencies from a bimodal distribution. In fact, their saddle-node bifurcation is part of a larger object called an $S$-curve which is discussed in \cite{golubitsky1985singularities} and illustrated in Figure \ref{fig:1b}. (Note that for values of $\sigma$ slightly less than one the $S$-curve in Figure \ref{fig:1b} results in bistability, namely, the existence of two distinct stable phase-locked solutions.)

\begin{figure}[H]
\captionsetup{font=scriptsize}
\centering
\begin{subfigure}{7cm}
  \centering
  \includegraphics[width=7cm]{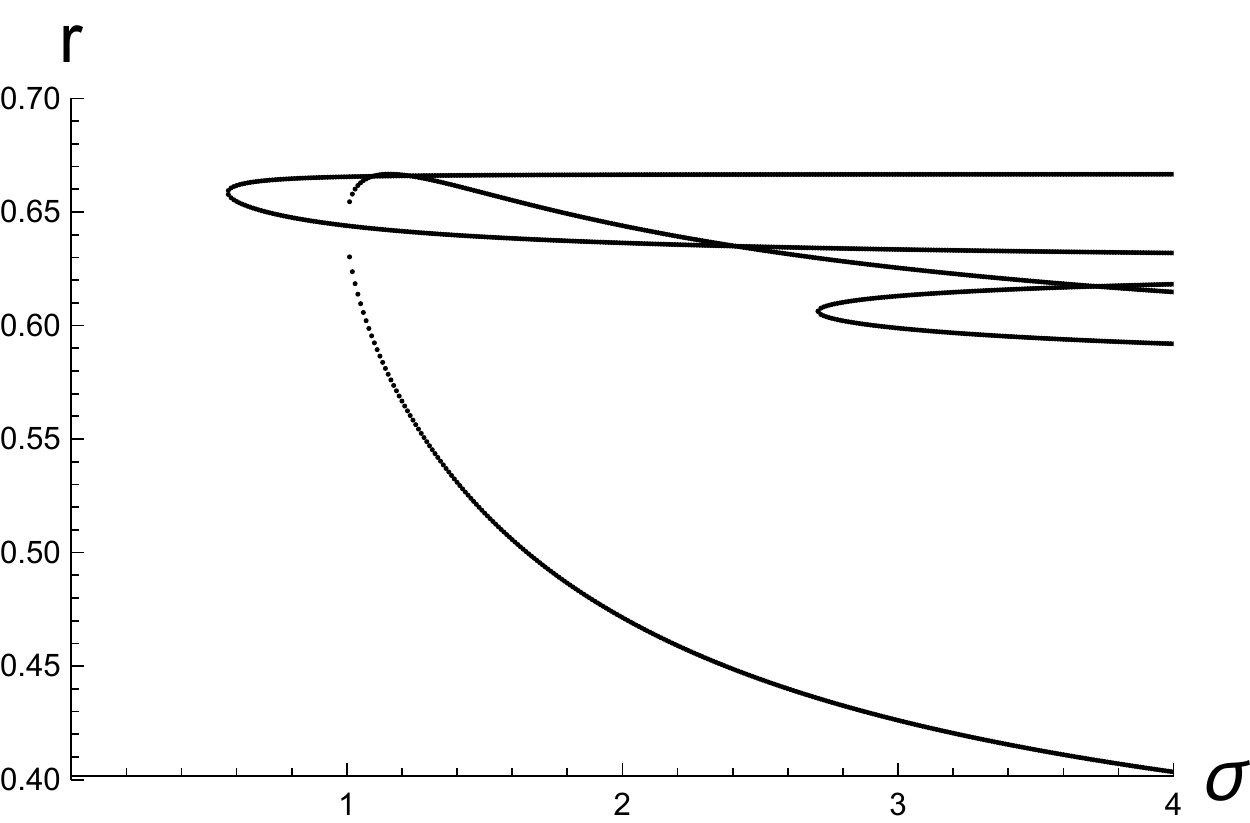}
  \caption{} \label{fig:1a}
\end{subfigure}
\begin{subfigure}{7cm}
  \centering
  \includegraphics[width=7cm]{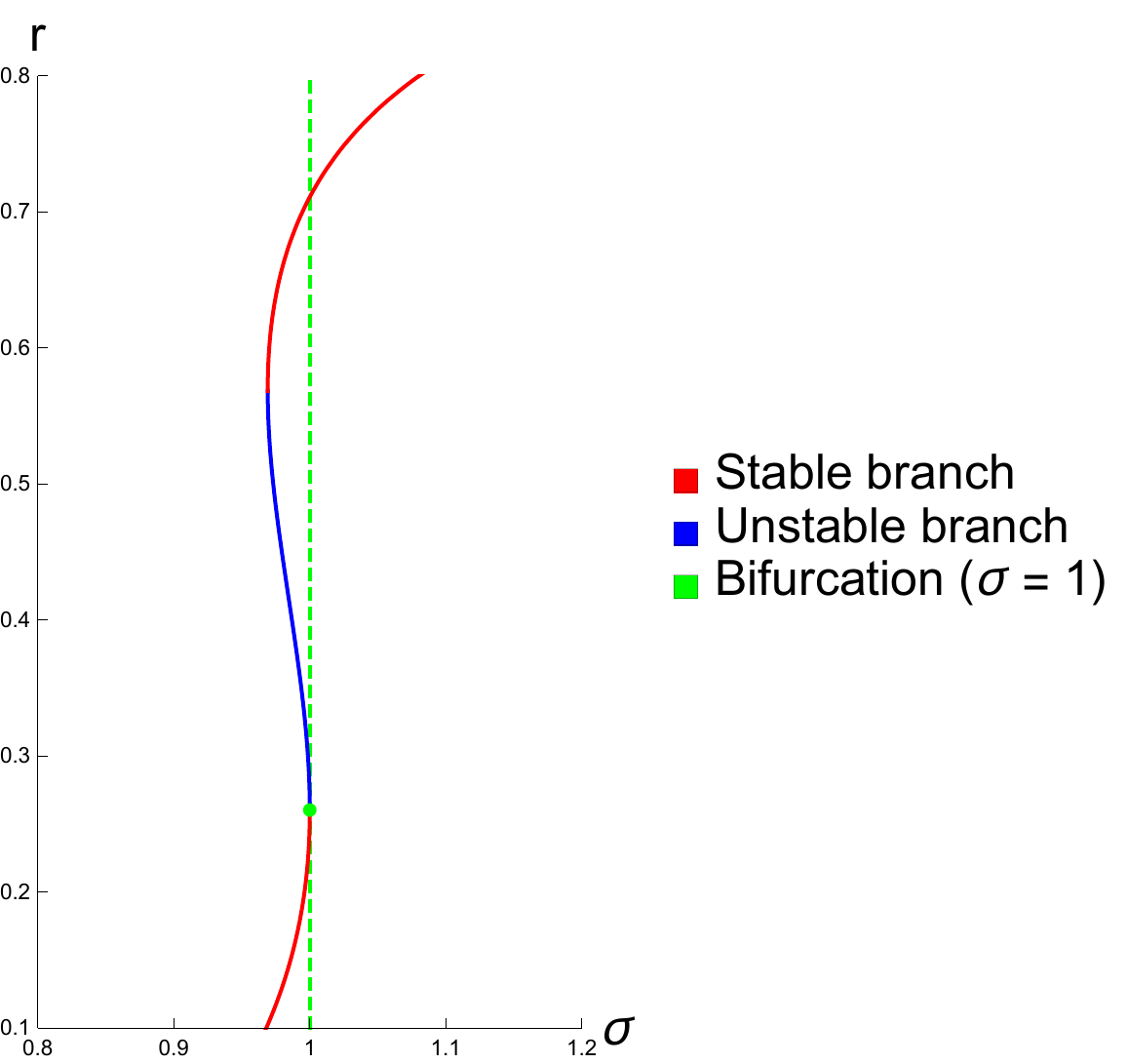}
  \caption{} \label{fig:1b}
\end{subfigure}
\caption{(A) A plot of all phase-locked solutions of the Kuramoto model \eqref{eq:model} with $(\omega_1,\omega_2,\omega_3) = (1,0,-1)$ and $(\gamma_1,\gamma_2,\gamma_3) = (1,1,1)$ for a range of values of $\sigma$. A phase-locked solution $(\theta_1,\theta_2,\theta_3)$ is represented by its order parameter $r$. (B) A plot of phase-locked solutions of the Kuramoto model \eqref{eq:model} with $(\omega_1,\omega_2,\omega_3,\omega_4,\omega_5) = (-0.69,0.69,0,0,0)$ and $(\gamma_1,\gamma_2,\gamma_3,\gamma_4,\gamma_5) = (0.20,1,1,1,1)$ for a range of values of $\sigma$. The plot contains a local subcritical saddle-node bifurcation at $\sigma = 1$ which is a part of a global $S$-curve. For values of $\sigma$ slightly less than one there are two distinct stable phase-locked solutions.} \label{fig:1}
\end{figure}

In Section \ref{sec:notation} we give a criterion for the Kuramoto model \eqref{eq:model} to have a bifurcation in which two branches of phase-locked solutions collide as $\sigma$ increases. We conjecture that this bifurcation is generically a subcritical saddle-node bifurcation. In addition, we give a sufficient condition for one of these branches to consist of stable phase-locked solutions. We end the section by applying this criterion to prove that for every $n \ge 3$ there exists values of $\omegab$ and $\gammab$ for which the Kuramoto model \eqref{eq:model} exhibits this bifurcation. We also show for $n \ge 5$ that this can be done so that one branch consists of stable phase-locked solutions. We end this section by presenting our methods as a constructive algorithm.

In Section \ref{sec:future} we discuss additional questions and future work. We defer proofs to the appendix.



\section{Notation and Main Results} \label{sec:notation}

In this section we introduce all relevant notations and definitions as well as state our main results. In general we defer the proofs of our results to the appendix.

Because we are only considering ring networks we let $e_i = (i,i+1)$ denote the $i$th edge in the network, $\eta_i = \theta_{i+1} - \theta_i$ the angle difference along the $i$th edge, and $\gamma_i = \gamma_{i+1,i}$ the coupling strength along the $i$th edge. For simplicity we further define the vector quantities $\omegab = (\omega_1, \dots, \omega_n)$, $\gammab = (\gamma_1, \dots, \gamma_n)$, $\thetab = (\theta_1, \dots, \theta_n)$, and $\etab = (\eta_1, \dots, \eta_n)$. With this notation we can rewrite \eqref{eq:model} as
\begin{align} \label{eq:model2}
\frac{d\theta_i}{dt} = \omega_i + \sigma (\gamma_i \sin \eta_i - \gamma_{i-1} \sin \eta_{i-1}) \quad \text{for} \quad i \in \{1,\dots,n\}.
\end{align}
This leads to the following definitions.
\begin{definition} \label{def:gJp}
Define the vector-valued function $\gb(\cdot,\gammab) : \R^n \rightarrow \R^n$ by
\begin{align*}
\gb(\thetab,\gammab)_i := -\gamma_i \sin \eta_i + \gamma_{i-1} \sin \eta_{i-1}
\end{align*}
and define the matrix-valued Jacobian $\Jb(\cdot,\gammab) : \R^n \rightarrow \R^{n \times n}$ by
\begin{align*}
\Jb(\thetab,\gammab)_{ij} := \frac{\partial}{\partial \theta_j} \gb(\thetab,\gammab)_i =
\begin{cases}
-\gamma_i \cos \eta_i & \mbox{if $j = i+1$,}
\\
\gamma_i \cos \eta_i + \gamma_{i-1} \cos \eta_{i-1}  & \mbox{if $j = i$,}
\\
-\gamma_{i-1} \cos \eta_{i-1} & \mbox{if $j = i-1$,}
\\
0 & \mbox{otherwise.}
\end{cases}
\end{align*}
Furthermore, let $\Pb$ be an $(n-1) \times n$ matrix with rows which form an orthonormal basis for the orthogonal complement $\1^\perp$ where $\1 = (1,\dots,1)^\top$.
\end{definition}

With this notation we can rewrite \eqref{eq:model} as the vector equation
\begin{align}
\frac{d\thetab}{dt} = \omegab - \sigma \gb(\thetab,\gammab).
\end{align}
Note that we chose the sign of $\gb(\thetab,\gammab)$ to simplify computations later on. Now a phase-locked solution has the form $\thetab(t) = \omegao \1 t + \thetab_0$ where $\omegao = \frac{1}{n} \sum_{i=1}^n \omega_i$, the mean of $\omegab$, is the common frequency of rotation and $\thetab_0$ represents the fixed angle differences. By rotating the reference frame we can assume that $\omegao = 0$ without loss of generality. Therefore a phase-locked solution is just a fixed point of the vector equation $\sigma \gb(\thetab,\gammab) = \omegab$. We will call a phase-locked solution stable if it is stable under all mean-zero perturbations, namely, perturbation orthogonal to $\1$. This is because $\gb(\cdot,\gammab)$ is inherently invariant under translations by the vector $\1$. From our choice of sign, we conclude that a phase-locked solution $\thetab$ is stable if the eigenvalues $\lambda_1,\dots,\lambda_n$ of the Jacobian matrix $\Jb(\thetab,\gammab)$ are all positive with the exception of $\lambda_1 = 0$ corresponding to the eigenvector $\1$. Finally, we note that $\Jb(\thetab,\gammab)$ is a Laplacian matrix, namely, a symmetric matrix with mean zero rows. Laplacian matrices form an important class of matrices and have an important connection to networks. One of the most famous connections between Laplacian matrices and networks is the Matrix Tree Theorem \cite{MR3097651}. We now state Theorem \ref{thm:detred} which is the special case of the famous Matrix Tree Theorem for ring networks.

\begin{theorem}[Matrix Tree Theorem] \label{thm:detred}
The reduced determinant satisfies the identities
\begin{align*}
\det_\red(\Jb(\thetab,\gammab)) := \det(\Pb \Jb(\thetab,\gammab) \Pb^\top) = n \sum_{i=1}^n \prod_{j \ne i} \gamma_j \cos \eta_j = n \left( \prod_{j=1}^n \gamma_j \cos \eta_j \right) \sum_{i=1}^n \frac{1}{\gamma_i \cos \eta_i} = \prod_{i=2}^n \lambda_i
\end{align*}
where $\lambda_1, \dots, \lambda_n$ are the eigenvalues of $\Jb(\thetab,\gammab)$ and $\lambda_1 = 0$ corresponds to the eigenvector $\1$.
\end{theorem}

We require one more definition before stating our first theorem.

\begin{definition} \label{def:vn}
Define the vector fields $\vb(\cdot,\gammab), \nb(\cdot,\gammab) : \R^n \rightarrow \R^n$ by
\begin{align}
\vb(\thetab,\gammab) := -\Pb^\top \adj(\Pb \Jb(\thetab,\gammab) \Pb^\top) \Pb \gb(\thetab,\gammab) \quad \text{and} \quad \nb(\thetab,\gammab) := \nabla_\thetab \det_\red(\Jb(\thetab,\gammab))
\end{align}
where $\adj$ is the adjugate matrix. Furthermore, define their inner product
\begin{align}
\delta(\thetab,\gammab) := \nb(\thetab,\gammab)^\top \vb(\thetab,\gammab).
\end{align}
\end{definition}

Now we can state our first main result Theorem \ref{thm:criterion1} which is a criterion for the Kuramoto model \eqref{eq:model} to have a bifurcation in which two branches of phase-locked solutions collide as we increase $\sigma$. We defer the proof of Theorem \ref{thm:criterion1} to the appendix.

\begin{theorem} \label{thm:criterion1}
Fix $\tb_0$ and $\gammab$ satisfying
\begin{align} \label{eq:hypothesis}
\det_\red(\Jb(\thetab_0,\gammab)) = 0 \quad \text{and} \quad \delta(\thetab_0,\gammab) < 0.
\end{align}
Furthermore, fix $\sigma_0$ and set $\omegab = \sigma_0 \gb(\thetab_0,\gammab)$. Then the Kuramoto model \eqref{eq:model} with parameters $\omegab$ and $\gammab$ has a bifurcation at $(\tb,\sigma) = (\tb_0,\sigma_0)$ which involves two branches of phase-locked solutions colliding as $\sigma$ increases to $\sigma_0$. Locally one branch satisfies $\det_\red(\Jb(\thetab,\gammab)) > 0$ and the other $\det_\red(\Jb(\thetab,\gammab)) <0$. In addition, if we suppose that the component-wise vector $\cos \etab_0$ has no zero components, then locally on each branch
\begin{align}
\text{$\#$positive eigenvalues of $\Jb(\thetab,\gammab)$} = n_+(\cos \etab_0) -
\begin{cases}
1 & \mbox{if $(-1)^{n-n_+(\cos \etab_0)} \det_\red(\Jb(\thetab,\gammab)) > 0$,}
\\
0 & \mbox{if $(-1)^{n-n_+(\cos \etab_0)} \det_\red(\Jb(\thetab,\gammab)) < 0$,}
\end{cases}
\end{align}
where $n_+(\cos \etab_0)$ is the number of positive components of $\cos \etab_0$. In particular, if $n_+(\cos \etab_0) = n-1$, then a branch of stable phase-locked solutions collide with a branch of 1-saddles.
\end{theorem}

We note that the bifurcation in Theorem \ref{thm:criterion1} is a subcritical saddle-node bifurcation if there are no other branches in the bifurcation. This occurs if for each sufficiently small $\epsilon > 0$ there exists a neighborhood $(\thetab_0,\sigma_0) \in U_\epsilon \subseteq \R^n$ for which
\begin{align} \label{eq:local}
\{ (\thetab,\sigma) \in U_\epsilon : \sigma \gb(\thetab,\gammab) = \omegab \} = \{ (\thetab(s), \sigma(s)) \in \R^n : |s| < \epsilon \}.
\end{align}
Now since $\vb(\thetab_0,\gammab) \ne \0$ we see that \eqref{eq:local} holds if and only if $(\thetab_0,\sigma_0)$ is an isolated zero of the analytic variety $V = \{ (\thetab,\sigma) \in \R^{n+1} : \sigma \gb(\thetab,\gammab) = \omegab \text{ and } \det_\red(\Jb(\thetab,\gammab)) = 0 \}$. Of course we mean that $(\thetab_0,\sigma_0)$ is an isolated zero up to translation of $\thetab_0$ by multiples of $\1$. This is true in all of our examples and we conjecture that it is true in generically. Heuristically, once we remove translation of $\thetab$ by $\1$, $V$ is actually an analytic variety in $n$ variables. Furthermore, $\sigma \gb(\thetab,\gammab) = \omegab$ contributes $n-1$ linearly independent equations since $\1^\top \gb(\thetab,\gammab) = \1^\top \omegab = \0$ for all values of $\thetab$ and $\gammab$. Therefore $V$ represents an analytic variety in $n$ variables and $n$ equations. For an analysis of the Kuramoto model using algebraic geometry see \cite{2015Chaos..25e3103M}.

Now Theorem \ref{thm:criterion1} is a consequence of Lemma \ref{lem:flow}.

\begin{lemma} \label{lem:flow}
Fix $\thetab_0$, $\gammab$, and $\sigma_0$ and set $\omegab = \sigma_0 \gb(\thetab_0,\gammab)$. Then there exists a unique analytic vector-valued function $\thetab(\cdot) = \thetab(\cdot, \thetab_0,\gammab) : \R \rightarrow \R^n$ satisfying
\begin{align} \label{eq:flow}
\frac{d \thetab}{ds} = \vb(\tb(s),\gammab)
\end{align}
with $\thetab(0) = \thetab_0$. Now define the analytic scalar-valued function $\sigma(\cdot) = \sigma(\cdot,\thetab_0,\gammab, \sigma_0) : \R \rightarrow \R_{>0}$ by
\begin{align} \label{eq:sigma}
\sigma(s) = \sigma_0 \exp \left(\int_0^s \det_\red(\Jb(\thetab(\st),\gammab)) d\st \right).
\end{align}
Then we have the following properties:
\begin{enumerate}
\item $\thetab(s)$ is a phase-locked solution of the Kuramoto model \eqref{eq:model} with fixed parameters $\omegab$ and $\gammab$ and coupling $\sigma(s)$ i.e.
\begin{align} \label{eq:fixed}
\sigma(s) \gb(\thetab(s),\gammab) = \omegab
\end{align}
for all $s \in \R$.
\item $\frac{d \sigma}{ds}$ has the same sign as $\det_\red(\Jb(\thetab(s),\gammab))$.
\end{enumerate}
\end{lemma}

Essentially, the conditions \eqref{eq:hypothesis} in Theorem \ref{thm:criterion1} cause $\sigma(s)$ in Lemma \ref{lem:flow} to have a local maximum at $\sigma(0) = \sigma_0$. We defer the details to the appendix, but give an example illustrating Theorem \ref{thm:criterion1} and Lemma \ref{lem:flow}. We start our example by setting
\begin{align} \label{eq:3}
\thetab_0 = (0.86727,1.84593,3.88114) \quad \text{and} \quad \gammab = (0.55284,1,1).
\end{align}
These values were obtained by using Mathematica to minimize $\delta(\thetab,\gammab)$ subject to the conditions $\det_\red(\Jb(\thetab,\gammab)) \\ = 0$ and $\gammab \in [0,1]^3$. In fact we have the value $\delta(\thetab_0,\gammab) = -0.37347$. In Figures \ref{fig:2a} and \ref{fig:2b} we plot the functions $\thetab(s)$ and $\sigma(s)$ from Lemma \ref{lem:flow}. In Figure \ref{fig:2c} we project the level set $\det_\red(\Jb(\thetab,\gammab)) = 0$, the normal vector $\nb(\thetab,\gammab)$, and the orbit of $\thetab(s)$ onto the mean-zero plane $\1^\perp$. Note that the orbit of $\thetab(s)$ just barely crosses the level set because it makes an obtuse angle (greater than a right angle) with the normal vector. This crossing results in the bifurcation and can also been seen in the local maximum of $\sigma(s)$ at $s = 0$.

\begin{figure}[H]
\captionsetup{font=scriptsize}
\centering
\begin{subfigure}{7cm}
  \centering
  \includegraphics[width=7cm]{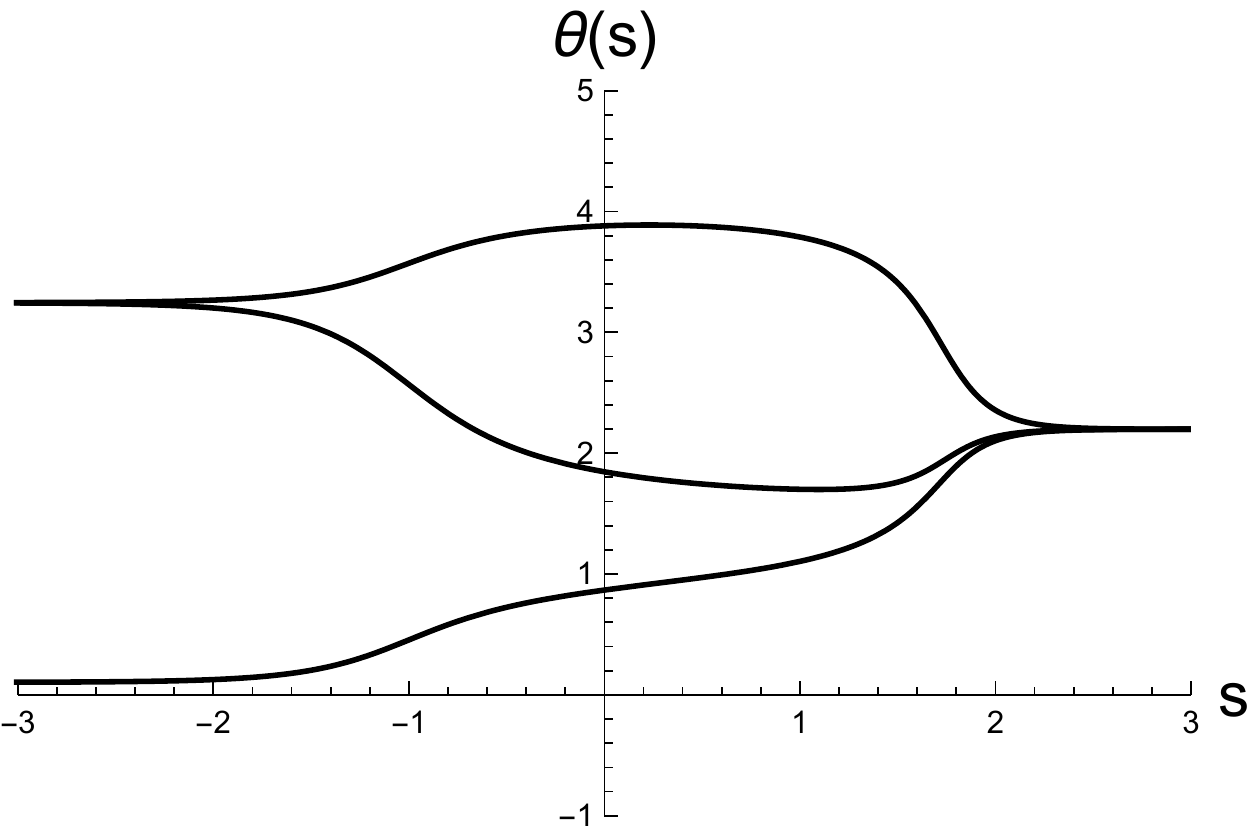}
  \caption{} \label{fig:2a}
\end{subfigure}
\begin{subfigure}{7cm}
  \centering
  \includegraphics[width=7cm]{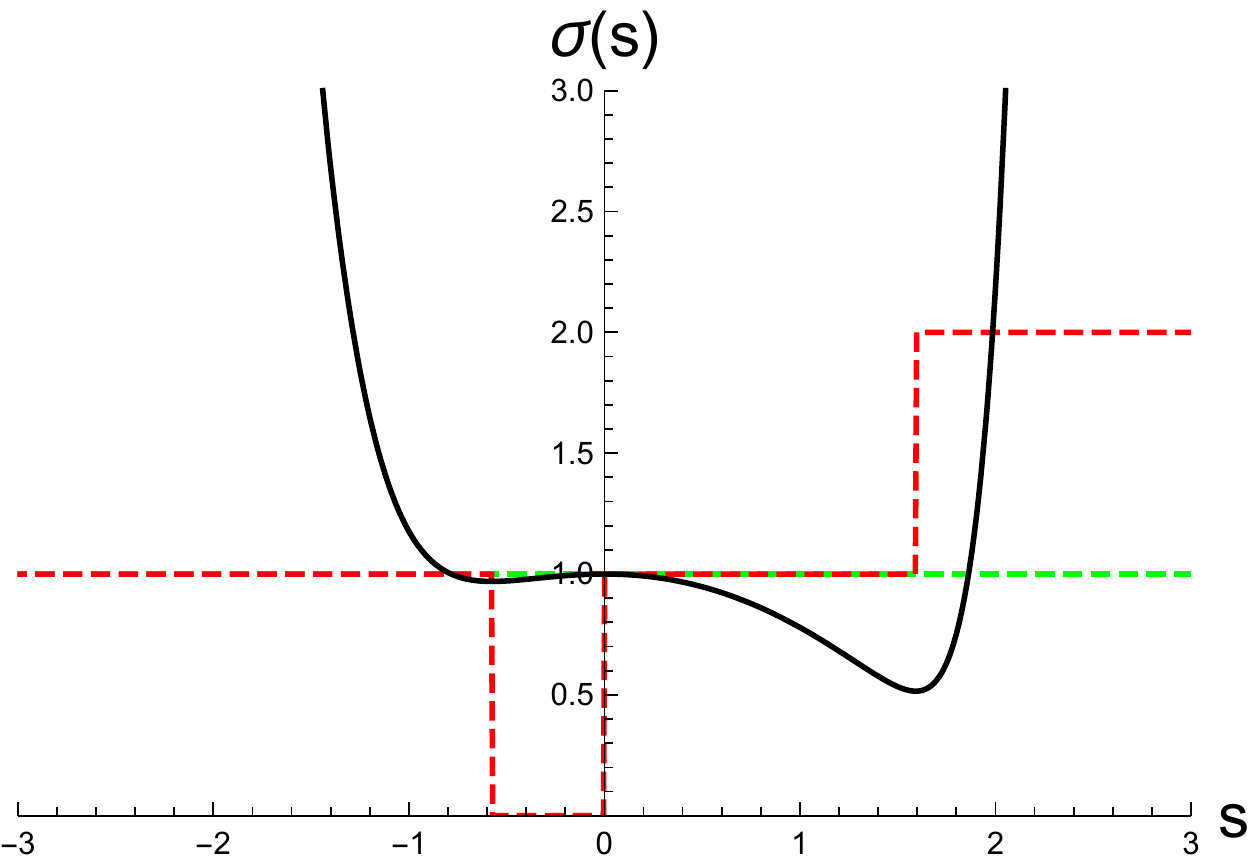}
  \caption{} \label{fig:2b}
\end{subfigure}
\medskip
\begin{subfigure}{7cm}
  \centering
  \includegraphics[width=7cm]{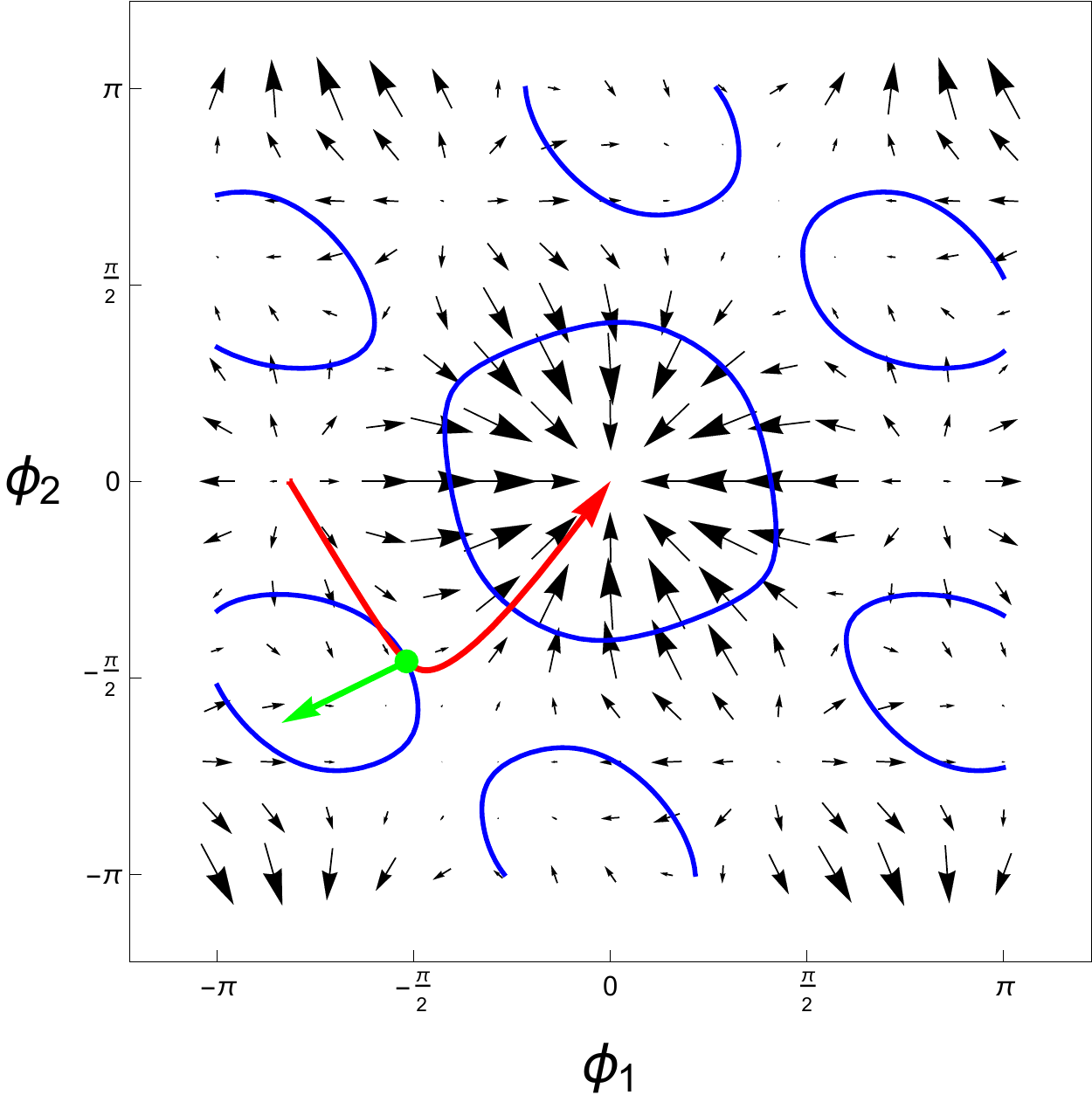}
  \caption{} \label{fig:2c}
\end{subfigure}
\caption{(A) A plot of $\thetab(s)$. (B) A plot of $\sigma(s)$. The dashed red lines indicate the number of positive eigenvalues of $\Jb(\thetab(s),\gammab)$. The dashed green line is simply to help see that $\sigma(s)$ has a local maximum at $s = 0$.(C) We projected $\R^3$ onto the mean-zero plane $\1^\perp$ via the change of variables $\phib = \Pb \thetab$. The blue ``circles" are connected components of the level set $\det_\red(\Jb(\thetab,\gammab)) = 0$, the green arrow is the normal vector $\nb(\thetab,\gammab)$, and the red curve is the orbit of $\thetab(s)$.} \label{fig:2}
\end{figure}

Now we give another example. This time however our bifurcation will have a branch consisting of stable phase-locked solutions. In fact, this is the bifurcation shown in Figure \ref{fig:1b}. Again we start by setting
\begin{align} \label{eq:5}
\thetab_0 = (3.80063,1.47760,6.77075,5,78071,4.79067) \quad \text{and} \quad \gammab = (0.20075,1,1,1,1).
\end{align}
Again, these values were obtained by using Mathematica to minimize $\delta(\thetab,\gammab)$ but now subject to the conditions $a_1(\thetab,\gammab) = 0$, $a_i(\thetab,\gammab) \ge 0$ for $2 \le i \le 4$, and $\gammab \in [0,1]^5$ where $\det(\lambda \Ib - \Jb(\thetab,\gammab)) = \lambda^5 + \sum_{i=1}^4 a_i(\thetab,\gammab) (-\lambda)^i$. In fact we have the value $\delta(\thetab_0,\gammab) = -0.70959$. Note that $a_0(\thetab,\gammab) = \det(\Jb(\thetab,\gammab)) = 0$ identically and that $a_1(\thetab,\gammab) = 0$ is equivalent to $\det_\red(\Jb(\thetab,\gammab)) = 0$. We also note that these conditions on the characteristic polynomial imply that $\Jb(\thetab,\gammab)$ has two zero eigenvalues (one trivial and one non-trivial) and that the rest are positive. In Figures \ref{fig:3a} and \ref{fig:3b} we again plot $\thetab(s)$ and $\sigma(s)$ from Lemma \ref{lem:flow}. Again we see that $\sigma(s)$ has a local maximum at $s = 0$ and that on one branch $\Jb(\thetab,\gammab)$ has four positive eigenvalues corresponding to stable phase-locked solutions. 

\begin{figure}[H]
\captionsetup{font=scriptsize}
\centering
\begin{subfigure}{7cm}
  \centering
  \includegraphics[width=7cm]{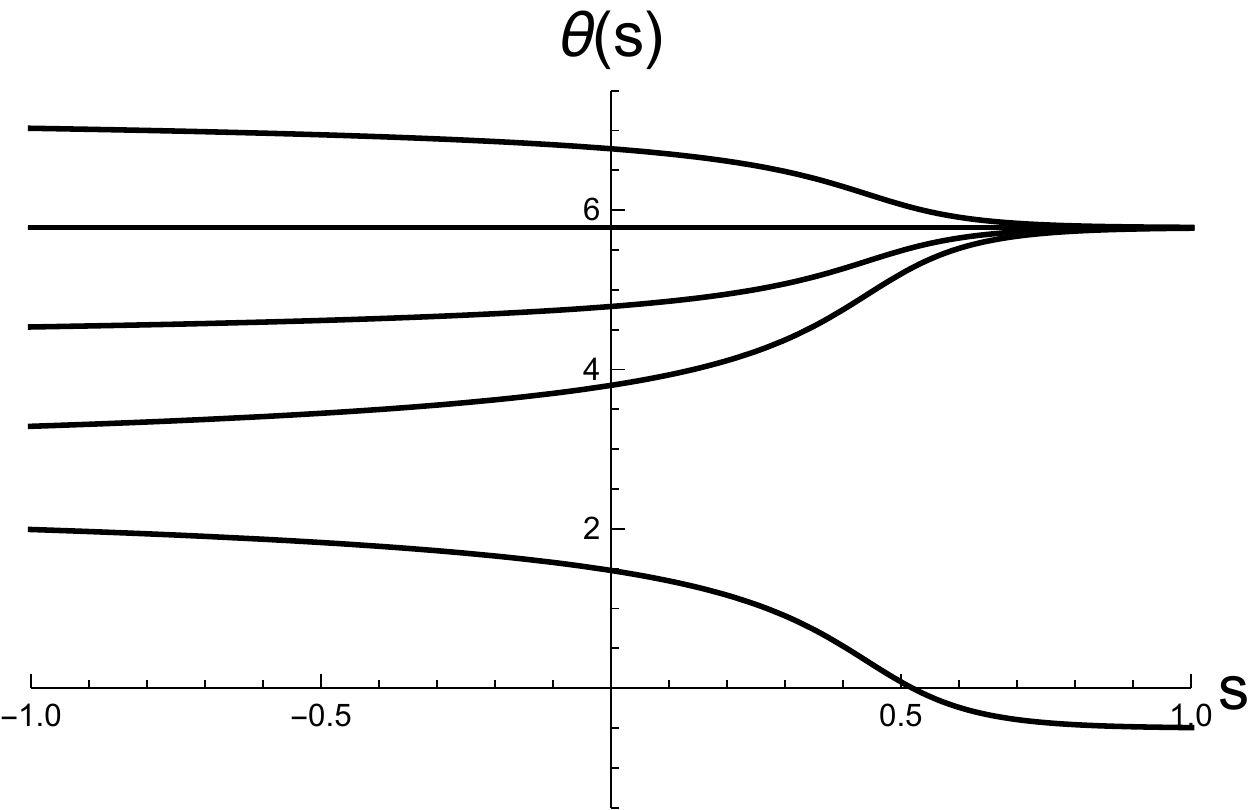}
  \caption{} \label{fig:3a}
\end{subfigure}
\begin{subfigure}{7cm}
  \centering
  \includegraphics[width=7cm]{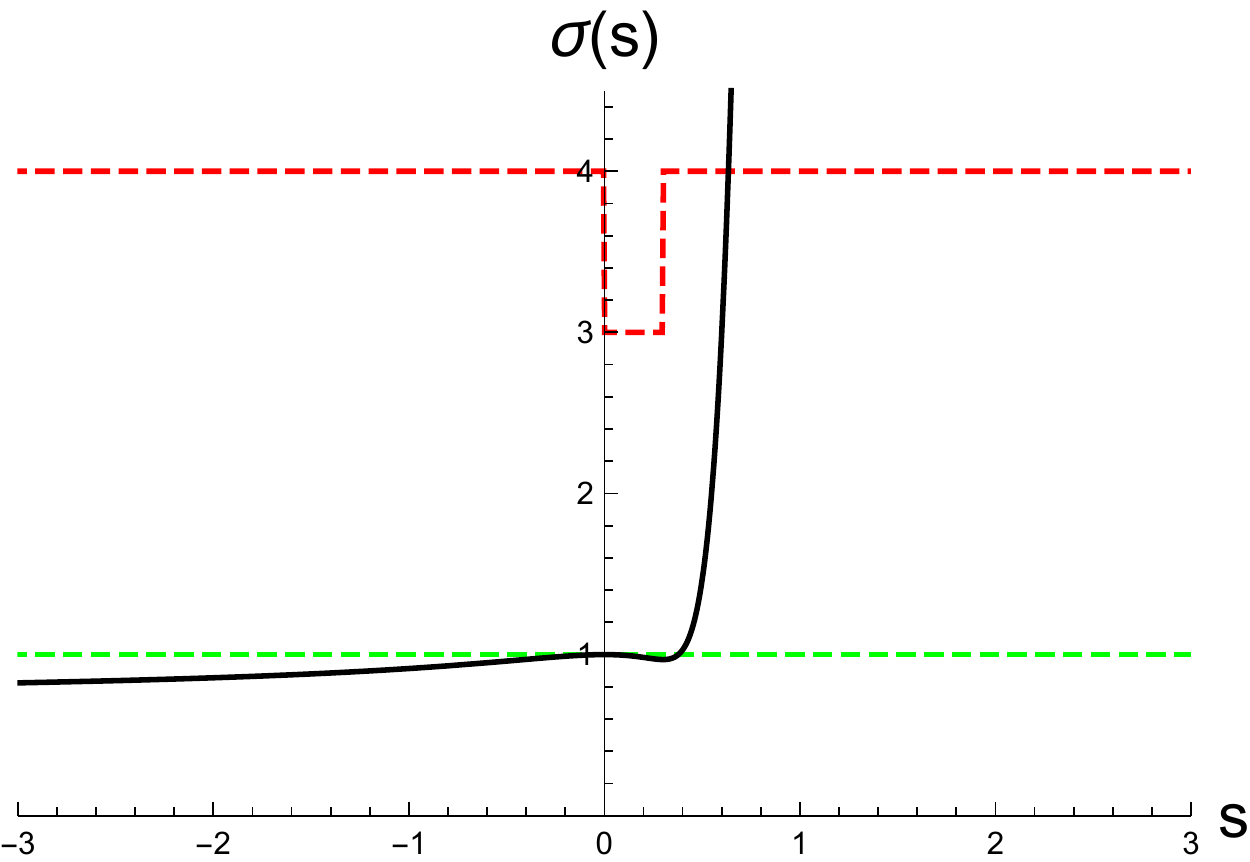}
  \caption{} \label{fig:3b}
\end{subfigure}
\caption{(A) A plot of $\thetab(s)$. (B) A plot of $\sigma(s)$. The dashed red lines indicate the number of positive eigenvalues of $\Jb(\thetab(s),\gammab)$. The dashed green line is simply to help see that $\sigma(s)$ has a local maximum at $s = 0$.} \label{fig:3}
\end{figure}

For our next main result we derive a condition on $\thetab_0$ that guarantees the existence of a $\gammab$ satisfying the conditions \eqref{eq:hypothesis} in Theorem \ref{thm:criterion1}. Before we can do this, however, we require some definitions and supporting lemmas. Again we defer the proofs to the appendix.

\begin{definition} \label{def:h}
Define the vector-valued function $\hb : \R^n \rightarrow \R^n$ by
\begin{align} \label{eq:h}
\h(\x)_i = \prod_{j \ne i} x_j.
\end{align}
\end{definition}

\begin{definition} \label{def:abc}
Define the three $n \times n$ matrices $\Ab(\thetab)$, $\B$, and $\Cb(\thetab)$ by
\begin{align*}
\Ab(\thetab)_{ij} := n \frac{\partial}{\partial \theta_j} \hb(\cos \etab)_i
\quad
\Bb_{ij} :=
\begin{cases}
1 & \mbox{if $i = j+1$,}
\\
-1 & \mbox{if $i = j$,}
\\
0 & \mbox{otherwise,}
\end{cases}
\quad
\Cb(\thetab)_{ij} =
\begin{cases}
n \left( \sum_{k \ne i} \frac{\sin \eta_k}{\cos \eta_k} \right) \hb(\cos \etab)_j & \mbox{if $i = j$,}
\\
-n \frac{\sin \eta_i}{\cos \eta_i} \hb(\cos \etab)_j & \mbox{if $i \ne j$.}
\end{cases}
\end{align*}
Note that $\Bb$ is the incidence matrix for a ring network with $n$ vertices and $n$ oriented edges $e_i = (i,i+1)$. Furthermore, define the two $n \times n$ matrices $\Sb(\thetab)$ and $\Tb(\thetab)$ by
\begin{align}
\Sb(\thetab) = n^2 \hb(\cos \etab)^\top \hb(\cos \etab) \quad \text{and} \quad \Tb(\thetab) = \Ab(\thetab) (\Bb^\top)^{-1} \Cb(\thetab)
\end{align}
where $(\Bb^\top)^{-1}$ is the pseudo-inverse of $\Bb^\top$.
\end{definition}

Continuing our example for $n = 3$ we have that
\begin{gather*}
\Ab(\thetab) =
\begin{pmatrix}
-\cos \eta _ 2 \sin \eta _ 3 & \cos \eta _ 3 \sin \eta _ 2 & \cos \eta _ 2 \sin \eta _ 3-\cos \eta _ 3 \sin \eta _ 2
\\
\cos \eta _ 3 \sin \eta _ 1-\cos \eta _ 1 \sin \eta _ 3 & -\cos \eta _ 3 \sin \eta _ 1 & \cos \eta _ 1 \sin \eta _ 3
\\
 \cos \eta _ 2 \sin \eta _ 1 & \cos \eta _ 1 \sin \eta _ 2-\cos \eta _ 2 \sin \eta _ 1 & -\cos \eta _ 1 \sin \eta _ 2
\end{pmatrix},
\\
\Bb =
\begin{pmatrix}
-1 & 0 & 1
\\
1 & -1 & 0
\\
0 & 1 & -1
\end{pmatrix},
\\
\Cb(\thetab) =
\begin{pmatrix}
3 \cos \eta _ 3 \sin \eta _ 2+3 \cos \eta _ 2 \sin \eta _ 3 & -3 \cos \eta _ 3 \sin \eta _ 1 & -3 \cos \eta _ 2 \sin \eta _ 1
\\
-3 \cos \eta _ 3 \sin \eta _ 2 & 3 \cos \eta _ 3 \sin \eta _ 1+3 \cos \eta _ 1 \sin \eta _ 3 & -3 \cos \eta _ 1 \sin \eta _ 2
\\
-3 \cos \eta _ 2 \sin \eta _ 3 & -3 \cos \eta _ 1 \sin \eta _ 3 & 3 \cos \eta _ 2 \sin \eta _ 1+3 \cos \eta _ 1 \sin \eta _ 2
\end{pmatrix}.
\end{gather*}

The motivation for Definitions \ref{def:h} and \ref{def:abc} are the quadratic form identities in Lemma \ref{lem:quadratic}.

\begin{lemma} \label{lem:quadratic}
We have the quadratic form identities
\begin{align}
\det_\red(\Jb(\thetab,\gammab))^2 = \hb(\gammab)^\top \Sb(\thetab) \hb(\gammab) \quad \text{and} \quad \delta(\thetab,\gammab) = \hb(\gammab)^\top \Tb(\thetab) \hb(\gammab).
\end{align}
Furthermore, the restriction $\hb : \R_{>0}^n \rightarrow \R_{>0}^n$ is a bijection with inverse
\begin{align} \label{eq:hinverse}
(\hb(\xb)^{-1})_i = \frac{1}{x_i} \left( \prod_{j=1}^n x_j \right)^{\frac{1}{n-1}}.
\end{align}
\end{lemma}

Next, we combine our matrices $\Sb(\thetab)$ and $\Tb(\thetab)$ in Definition \ref{def:r}.

\begin{definition} \label{def:r}
Define the parameterized family of symmetric $n \times n$ matrices $\Rb_\tau(\thetab)$ by
\begin{align}
\Rb_\tau(\thetab) = \Sb(\thetab) \tau + \frac{1}{2}(\Tb(\thetab) + \Tb(\thetab)^\top).
\end{align}
\end{definition}

The motivation for Definition \ref{def:r} is the identity
\begin{align} \label{eq:quadratic}
\hb(\gammab)^\top \Rb_\tau(\thetab) \hb(\gammab) = \det_\red(\Jb(\thetab,\gammab))^2 \tau + \delta(\thetab,\gammab)
\end{align}
which easily follows from Lemma \ref{lem:quadratic}. We now make the observation that if $\thetab_0$ and $\gammab$ satisfy \eqref{eq:hypothesis}, then $\xb^\top \Rb_\tau(\thetab_0) \xb < 0$ for all $\tau$ where $\xb = \hb(\gammab) \in \R_{>0}^n$ by Lemma \ref{lem:quadratic}. Conversely, we will show that if there exists vectors $\xb_\tau \in \R_{\ge 0}^n$ such that $\xb_\tau^\top \Rb_\tau(\thetab_0) \xb_\tau / \| \xb_\tau \|_1^2 < \epsilon < 0$ for all sufficiently large $\tau$, then there exists a vector $\xb \in \R_{>0}^n$ such that $\xb^\top \Rb_\tau(\theta_0) \xb < 0$. By Lemma 4 we can choose $\gammab$ such that $\xb = \hb(\gammab)$ in which case \eqref{eq:hypothesis} is satisfied. In order to demonstrate the existence of our vectors $\xb_\tau \in \R_{\ge 0}^n$ we require Lemma \ref{lem:min} which is a slight modification of a result of Kreps in \cite{KREPS1984105}. Before stating Lemma \ref{lem:min} we first make Definition \ref{def:min}.

\begin{definition} \label{def:min}
Let $\Rb$ be an $n \times n$ matrix and $\Ic \subseteq \{1, \dots, n\}$. Define $\Rb_\Ic$ to be the sub-matrix of $\Rb$ with rows and columns corresponding to the indices in $\Ic$. In other words, if $\Ic = \{ i_1 < \dots < i_m\}$, then $\Rb_\Ic$ is the $m \times m$ matrix with entries $(\Rb_\Ic)_{k\ell} = \Rb_{i_k i_\ell}$. Furthermore, given $i \in \Ic$, define $\Rb_{\Ic,i}$ to be the matrix $\Rb_\Ic$ with the $i$th row replaced by a row of all ones.
\end{definition}

\begin{lemma} \label{lem:min}
For any symmetric $n \times n$ matrix $\Rb$ we have the identity
\begin{align} \label{eq:min}
\min_{\xb \in \R_{\ge 0}^n} \frac{\xb^\top \Rb \xb}{\| \xb \|_1^2} = \min \left\{ \frac{\det(\Rb_\Ic)}{\sum_{i \in \Ic} \det(\Rb_{\Ic,i})} : \Ic \subseteq \{1,\dots,n\} \text{ satisfies } \min_{i \in \Ic} \det(\Rb_{\Ic,i}) > 0 \right\}
\end{align}
where $\| \xb \|_1 := \sum_{i=1}^n |x_i|$.
\end{lemma}

We note that $\det(\Rb_\tau(\thetab)_\Ic)$ and $\det(\Rb_\tau(\thetab)_{\Ic,i})$ are polynomials in $\tau$. In fact, they have degree at most one since $\Sb(\thetab)$ has rank one. Now in order to apply Lemma \ref{lem:min} we make Definition \ref{def:limit}.

\begin{definition} \label{def:limit}
Define
\begin{align*}
\Ic(\thetab) := \{ \Ic \subseteq \{1,\dots,n\} : \min_{i \in \Ic} \det(\Rb_\tau(\thetab)_{\Ic,i}) > 0 \text{ for all sufficiently large $\tau$} \}
\end{align*}
and for each $\Ic \in \Ic(\thetab)$ define
\begin{align*}
p_\Ic(\thetab) :=
\begin{cases}
1 & \mbox{if $\lim_{\tau \rightarrow \infty} \frac{\det(\Rb_\tau(\thetab)_\Ic)}{\sum_{i \in \Ic} \det(\Rb_\tau(\thetab)_{\Ic,i})} > 0$,}
\\
0 & \mbox{if $\lim_{\tau \rightarrow \infty} \frac{\det(\Rb_\tau(\thetab)_\Ic)}{\sum_{i \in \Ic} \det(\Rb_\tau(\thetab)_{\Ic,i})} = 0$,}
\\
-1 & \mbox{if $\lim_{\tau \rightarrow \infty} \frac{\det(\Rb_\tau(\thetab)_\Ic)}{\sum_{i \in \Ic} \det(\Rb_\tau(\thetab)_{\Ic,i})} < 0$,}
\end{cases}
\end{align*}
and
\begin{align*}
q_\Ic(\thetab) :=
\begin{cases}
1 & \mbox{if $\det(\Rb_\tau(\thetab)_\Ic) > 0$ for all sufficiently large $\tau$,}
\\
0 & \mbox{if $\det(\Rb_\tau(\thetab)_\Ic) = 0$ for all sufficiently large $\tau$,}
\\
-1 & \mbox{if $\det(\Rb_\tau(\thetab)_\Ic) < 0$ for all sufficiently large $\tau$,}
\end{cases}
\end{align*}
and set
\begin{align*}
\Delta_p(\thetab) := \min \{ p_\Ic(\thetab) : \Ic \in \Ic(\thetab) \} \quad \text{and} \quad \Delta_q(\thetab) := \min \{ q_\Ic(\thetab) : \Ic \in \Ic(\thetab) \}.
\end{align*}
\end{definition}

The motivation for the definitions of $\Delta_p(\thetab)$ and $\Delta_q(\thetab)$ is that they in some sense model the sign of \eqref{eq:min} for $\Rb_\tau(\thetab)$ for large values of $\tau$. We plot $\Delta_p(\thetab)$ and $\Delta_q(\thetab)$ in Figure \ref{fig:4} for $n = 3$. Now we can state our next main result Theorem \ref{thm:criterion2}.

\begin{theorem} \label{thm:criterion2}
There exists a $\gammab$ satisfying \eqref{eq:hypothesis} if $\Delta_p(\thetab_0) = -1$. Conversely, \eqref{eq:hypothesis} is not satisfied for any $\gammab$ if $\Delta_q(\thetab_0) \in \{0,1\}$.
\end{theorem}

We make a brief comment on the strength of Theorem \ref{thm:criterion2}. Now
\begin{align} \label{eq:exhaustive}
 \{ \thetab \in \R^n : \Delta_p(\thetab) \ne \Delta_q(\thetab) \} \subseteq \bigcup_{\Ic \subseteq \{1, \dots,n\}} \{ \thetab \in \R^n : \ell_\Ic(\thetab) = 0\}
\end{align}
where $\ell_\Ic(\thetab) = \lim_{\tau \rightarrow \infty} \frac{1}{\tau} \det(\Rb_\tau(\thetab)_\Ic)$. This is because $\det(\Rb_\tau(\thetab)_\Ic)$ and $\det(\Rb_\tau(\thetab)_{\Ic,i})$ are polynomials of degree at most one in $\tau$ and therefore $p_\Ic(\thetab) = q_\Ic(\thetab)$ if $\ell_\Ic(\thetab) \ne 0$. By properties of the determinant we know that $\ell_\Ic(\thetab)$ is a sum of determinants of matrices with entries from $\Sb(\thetab)$ and $\Tb(\thetab)$. Since the entries of $\Sb(\thetab)$ and $\Tb(\thetab)$ are analytic we easily deduce that $\ell_\Ic(\thetab)$ is also analytic. Therefore both sets in \eqref{eq:exhaustive} have $n$-dimensional Lebesgue measure zero if $\ell_\Ic(\thetab)$ is not identically zero for all $\Ic \subseteq \{1,\dots,n\}$. In particular, this would imply that the exceptional set of Theorem \ref{thm:criterion2} has $n$-dimensional Lebesgue measure zero. Unfortunately, we have yet to discover a proof that $\ell_\Ic(\thetab)$ is not identically zero due to the complexity of the definition of $\ell_\Ic(\thetab)$. However, we conjecture that this is the case and give some examples for $n = 3$. In particular,
\begin{gather*}
\ell_{\{1\}}(\thetab) =
\cos ^2\eta _2 \cos ^2\eta _3,
\\
\ell_{\{1,2\}}(\thetab) =
-6 \sin \eta _1 \sin \eta _2 \cos \eta _1 \cos \eta _2 \cos ^4\eta _3-3 \sin ^2\eta _1 \cos ^2\eta _2 \cos ^4\eta _3 -3 \sin ^2\eta _2 \cos ^2\eta _1 \cos ^4\eta _3-
\\
3 \sin \eta _1 \sin \eta _3 \cos \eta _1 \cos ^2\eta _2 \cos ^3\eta _3-3 \sin \eta _2 \sin \eta _3 \cos ^2\eta _1 \cos \eta _2 \cos ^3\eta _3,
\\
\ell_{\{1,2,3\}}(\thetab) =
9 \sin \eta _2 \sin ^3\eta _3 \cos ^3\eta _2 \cos \eta _3 \cos ^4\eta _1+9 \sin ^3\eta _2 \sin \eta _3 \cos \eta _2 \cos ^3\eta _3 \cos ^4\eta _1+
\\
9 \sin \eta _1 \sin ^3\eta _2 \cos \eta _2 \cos ^4\eta _3 \cos ^3\eta _1+9 \sin \eta _1 \sin ^3\eta _3 \cos ^4\eta _2 \cos \eta _3 \cos ^3\eta _1+
\\
9 \sin ^3\eta _1 \sin \eta _2 \cos ^3\eta _2 \cos ^4\eta _3 \cos \eta _1+9 \sin ^3\eta _1 \sin \eta _3 \cos ^4\eta _2 \cos ^3\eta _3 \cos \eta _1+
\\
18 \sin ^2\eta _2 \sin ^2\eta _3 \cos ^2\eta _2 \cos ^2\eta _3 \cos ^4\eta _1+18 \sin ^2\eta _1 \sin ^2\eta _2 \cos ^2\eta _2 \cos ^4\eta _3 \cos ^2\eta _1+
\\
18 \sin ^2\eta _1 \sin ^2\eta _3 \cos ^4\eta _2 \cos^2\eta _3 \cos ^2\eta _1+45 \sin \eta _1 \sin \eta _2 \sin ^2\eta _3 \cos ^3\eta _2 \cos ^2\eta _3 \cos ^3\eta _1+
\\
45 \sin \eta _1 \sin ^2\eta _2 \sin \eta _3 \cos ^2\eta _2 \cos ^3\eta _3 \cos ^3\eta _1+ 45 \sin ^2\eta _1 \sin \eta _2 \sin \eta _3 \cos ^3\eta _2 \cos ^3\eta _3 \cos ^2\eta _1,
\end{gather*}
and
\begin{align*}
\ell_{\{1\}}(\thetab_0) = \frac{1}{16} \ne 0, \quad \ell_{\{1,2\}}(\thetab_0) = -\frac{27}{128} \ne 0, \quad \ell_{\{1,2,3\}}(\thetab_0) = \frac{2187}{4096} \ne 0,
\end{align*}
where $\thetab_0 = (0,\pi/3,2\pi/3)$. Of course the remaining functions $\ell_{\{2\}}(\thetab)$, $\ell_{\{3\}}(\thetab)$, $\ell_{\{1,3\}}(\thetab)$, and $\ell_{\{2,3\}}(\thetab)$ cannot be identically zero by symmetry.

Finally, we show that the two sets
\begin{align} \label{eq:set1}
\{ \thetab \in \R^n : \Delta_p(\thetab) = -1 \}
\end{align}
and
\begin{align} \label{eq:set2}
\{ \thetab \in \R^n : \text{$\Delta_p(\thetab) = -1$, $\cos \etab$ has no zero components, and $n_+(\cos \etab) = n-1$} \}
\end{align}
are non-empty for all $n \ge 3$ and $n \ge 5$ respectively. Therefore by applying Theorems \ref{thm:criterion1} and \ref{thm:criterion2} we deduce our final main result Theorem \ref{thm:existence}.

\begin{theorem} \label{thm:existence}
For every $n \ge 3$, there exists an $\omegab$ and $\gammab$ for which the Kuramoto model \eqref{eq:model} has a bifurcation which involves two branches of phase-locked solutions colliding as $\sigma$ increases. Furthermore, for every $n \ge 5$, there exists an $\omegab$ and $\gammab$ for which the Kuramoto model \eqref{eq:model} has a bifurcation which involves a branch of stable phase-locked solutions colliding with a branch of 1-saddles as $\sigma$ increases.
\end{theorem}

We include the proof of Theorem \ref{thm:existence} here because the proof is an application of an important property of the matrices $\Sb(\thetab)$ and $\Tb(\thetab)$ given in Lemma \ref{lem:st}.

\begin{lemma} \label{lem:st}
Let $1 \le n \le \nt$ and for any $\thetab = (\theta_1, \dots, \theta_n) \in \R^n$ set $\thetabt = (\theta_1, \dots, \theta_n, \theta_1, \dots, \theta_1) \in \R^{\nt}$. Then
\begin{align*}
\Sbt(\thetabt) =
\begin{pmatrix}
(\frac{\nt}{n})^2 \Sb(\thetab) & \hdots
\\
\vdots & \ddots
\end{pmatrix}
\quad \text{and} \quad
\Tbt(\thetabt) =
\begin{pmatrix}
(\frac{\nt}{n})^2 \Tb(\thetab) & \hdots
\\
\vdots & \ddots
\end{pmatrix}.
\end{align*}
In other words, $\Sb(\thetab)$ and $\Tb(\thetab)$ are positive multiples of the upper left $n \times n$ submatrices of $\Sbt(\thetabt)$ and $\Tbt(\thetabt)$ respectively.
\end{lemma}

\begin{proof}[Proof of Theorem \ref{thm:existence}]
We start with the first statement for $n \ge 3$ and set
\begin{align} \label{eq:nge3}
\thetabt_0 = (0.86727,1.84593,3.88114,0.86727,\dots,0.86727) \in \R^n.
\end{align}
Note that $\thetabt_0$ reduces to $\thetab_0$ in \eqref{eq:3} when $n = 3$. By Lemma \ref{lem:st} we know that
\begin{align*}
\det(\Rbt_\tau(\thetabt_0)_{\{1,2\}}) = \left(\frac{n}{3} \right)^4 \det(\Rb_\tau(\thetab_0)_{\{1,2\}}) = \left(\frac{n}{3} \right)^4 \left(-0.01174 n^4 \tau -0.01162 \right)
\end{align*}
and
\begin{align*}
\det(\Rbt_\tau(\thetabt_0)_{\{1,2\},i}) = \left(\frac{n}{3} \right)^2 \det(\Rb_\tau(\thetab_0)_{\{1,2\},i}) = \left(\frac{n}{3} \right)^2
\begin{cases}
0.55240 \tau -1.42429 & \mbox{if $i = 1$,}
\\
0.44330 \tau -1.53160 & \mbox{if $i = 2$.}
\end{cases}
\end{align*}
Therefore $\Delta_p(\thetabt_0) = -1$ for each $n \ge 3$ and the result follows from Theorem \ref{thm:criterion2}.

In a similar way we prove the second statement for $n \ge 5$ by setting
\begin{align} \label{eq:nge5}
\thetabt_0 = (3.80063,1.47760,6.77075,5,78071,4.79067,3.80063,\dots,3.80063) \in \R^n.
\end{align}
Again note that $\thetabt_0$ reduces to $\thetab_0$ in \eqref{eq:5} when $n = 5$. By Lemma \ref{lem:st} we know that
\begin{align*}
\det(\Rbt_\tau(\thetabt_0)_{\{1,2\}}) = \left(\frac{n}{5} \right)^4 \det(\Rb_\tau(\thetab_0)_{\{1,2\}}) = \left(\frac{n}{5} \right)^4\left( -0.00024 \tau + 0.00191 \right)
\end{align*}
and
\begin{align*}
\det(\Rbt_\tau(\thetabt_0)_{\{1,2\},i}) = \left(\frac{n}{5} \right)^2 \det(\Rb_\tau(\thetab_0)_{\{1,2\},i}) = \left(\frac{n}{5} \right)^2
\begin{cases}
0.02296 \tau -0.18050 & \mbox{if $i = 1$,}
\\
0.01844 \tau -0.15351 & \mbox{if $i = 2$.}
\end{cases}
\end{align*}
Therefore $\Delta_p(\thetabt_0) = -1$ for each $n \ge 5$ and $\cos \etabt_0 = (-0.68327,0.54866,0.54866,0.54866,0.54866,1,\dots,1)$. Therefore the result follows from Theorems \ref{thm:criterion1} and \ref{thm:criterion2}.
\end{proof}

For $3 \le n \le 4$ we used Mathematica to minimize $\delta(\thetab,\gammab)$ subject to the same conditions on the characteristic polynomial as those used to find $\thetab_0$ and $\gammab$ in \eqref{eq:5}. However, we were unable to obtain a negative value for $\delta(\thetab,\gammab)$. This gives numerical evidence that Theorem \ref{thm:existence} is best possible although we admit that we do not have a proof of this.

We end this section by briefly presenting Theorem \ref{thm:existence} as a constructive algorithm. For the first part of Theorem \ref{thm:existence} choose any $n \ge 3$ and complete the following steps:
\begin{enumerate}
\item Set $\thetab_0$ as in \eqref{eq:nge3},
\item Find a vector $\xb \in \R_{>0}^n$ such that $\hb(\cos \etab_0) \xb = 0$ and $\xb^\top \Tb(\thetab_0) \xb < 0$,
\item Set $\gammab = \hb(\xb)^{-1}$, $\sigma_0 = 1$, and $\omegab = \gb(\thetab_0,\gammab)$.
\end{enumerate}
The existence of the vector in (2) is guaranteed by the proofs of Theorems \ref{thm:criterion2} and \ref{thm:existence}. Then our choice of $\gammab$ results in the equations $\det_\red(\Jb(\thetab_0,\gammab)) = \hb(\cos \etab_0) \xb = 0$ and $\delta(\thetab_0,\gammab) = \xb^\top \Tb(\thetab_0) \xb < 0$. Therefore the Kuramoto model with parameters $\omegab$ and $\gammab$ has the desired bifurcation at $(\thetab,\sigma) = (\thetab_0,\sigma_0)$ by Theorem \ref{thm:criterion1}.

For the second part of Theorem \ref{thm:existence} choose any $n \ge 5$ and set $\thetab_0$ as in \eqref{eq:nge5}.



\section{Other Directions and Future Work} \label{sec:future}

As mentioned in the introduction our bifurcation in Figure \ref{fig:1a} is locally a subcritical saddle-node bifurcation and globally an $S$-curve. We conjecture that this is generically the case. (Note that we discussed the genericity of the subcritical saddle-node bifurcation heuristically after the statement of Theorem \ref{thm:criterion1}.

One may wonder if we can numerically generate higher order $S$-curves. This can occur if we have a hysteresis point which we then perturb. In our case a hysteresis point is a pair $(\thetab_0,\bg)$ for which the curve $\thetab(s)$ is tangential to the level set $\det_\red(\Jb(\thetab_0,\bg)) = 0$. In other words, the curve has a higher order of contact with the level set. If we can make this order of contact high enough we should be able to construct a higher order $S$-curve. To do this we set
\begin{align*}
\det_\red(\Jb(\thetab(s),\bg) = \sum_{k=0}^\infty a_k(\thetab_0,\bg) s^k
\end{align*}
and deduce that
\begin{align*}
a_0(\thetab,\bg) = \det_\red(\Jb(\thetab,\bg)) \quad \text{and} \quad  a_{k+1}(\thetab,\bg) = \langle \vb(\thetab,\bg), \nabla_\thetab a_k(\thetab,\bg) \rangle
\end{align*}
for $k \ge 0$. We say that the curve $\thetab(s)$ has order of contact $d$ with the level set $\det_\red(\Jb(\thetab,\bg)) = 0$ at $\thetab_0$ if $a_d(\thetab_0,\bg) \ne 0$ and $a_k(\thetab_0,\bg) = 0$ for $0 \le k < d$.

We executed this strategy and set $a_0(\thetab_0,\bg) = a_1(\thetab_0,\bg) = a_2(\thetab_0,\bg) = 0$ and numerically found the values
\begin{align*}
\thetab_0 = (5.00250, 3.78647, 2.09311, 4.35320, 5.58644, 0.99502),
\\
\bg = (0.87815, 1.22758, 2.45939, 0.51472, 6.30989, 2.91474).
\end{align*}
We then perturbed the value of $\thetab_0$ by
\begin{align*}
\epsilonb_1 = (0.04000, 0.02000, 0.01000, 0.01000, 0.01000, -0.02968),
\\
\epsilonb_2 = (0.02000, -0.02000, 0.04000, -0.03000, -0.01000, -0.16604),
\end{align*}
which leads to the values
\begin{align*}
\det_\red(\Jb(\thetab_0,\gammab)) &= 0 \quad \text{and} \quad \delta(\thetab_0,\gammab) = 0,
\\
\det_\red(\Jb(\thetab_0 + \epsilonb_1,\gammab)) &= 0 \quad \text{and} \quad \delta(\thetab_0 + \epsilonb_1,\gammab) = -6.18641,
\\
\det_\red(\Jb(\thetab_0 + \epsilonb_2,\gammab)) &= 0 \quad \text{and} \quad \delta(\thetab_0 + \epsilonb_2,\gammab) = 12.7292,
\end{align*}
and the three figures in Figure \ref{fig:4}.

\begin{figure}[H]
\captionsetup{font=scriptsize}
\centering
\begin{subfigure}{5cm}
  \centering
  \includegraphics[width=4.5cm]{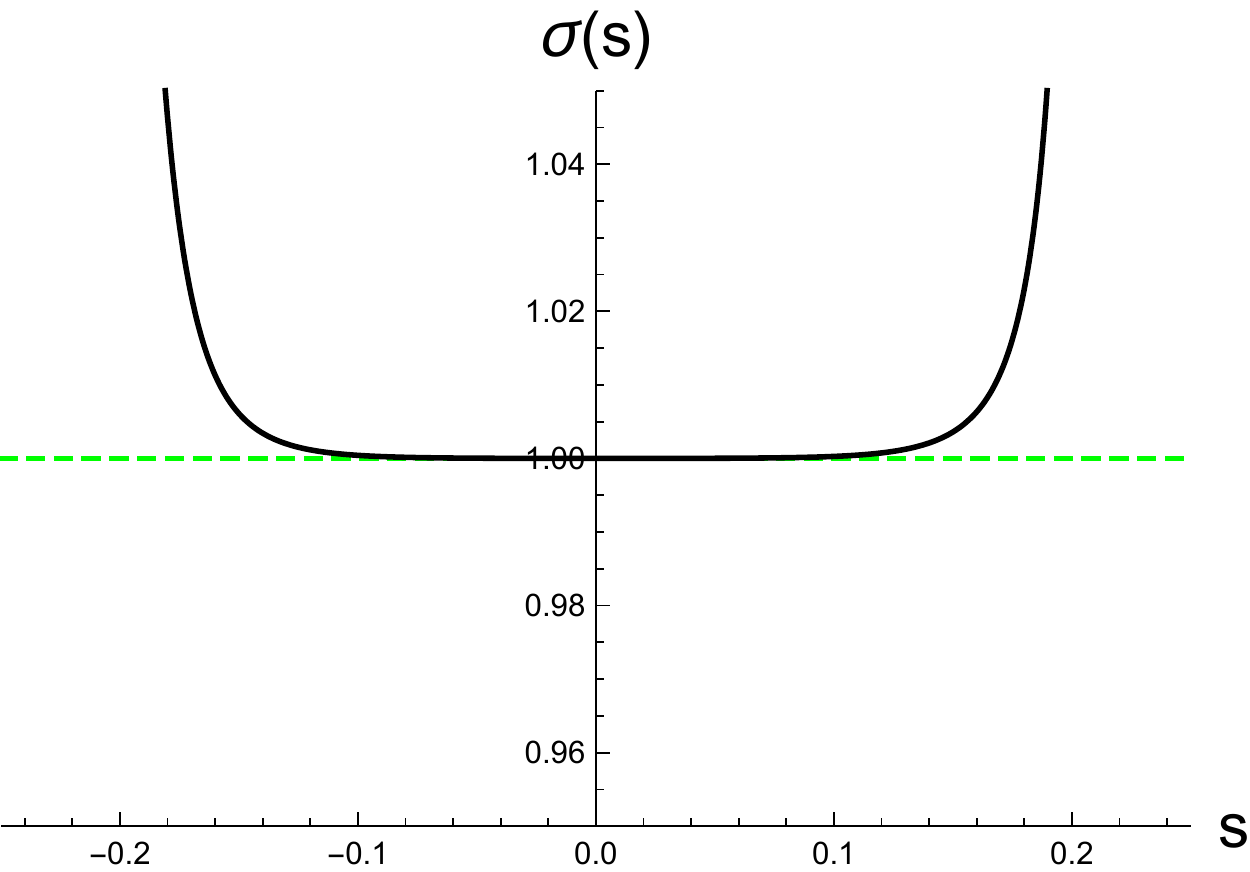}
  \caption{} \label{fig:4a}
\end{subfigure}
\begin{subfigure}{5cm}
  \centering
  \includegraphics[width=4.5cm]{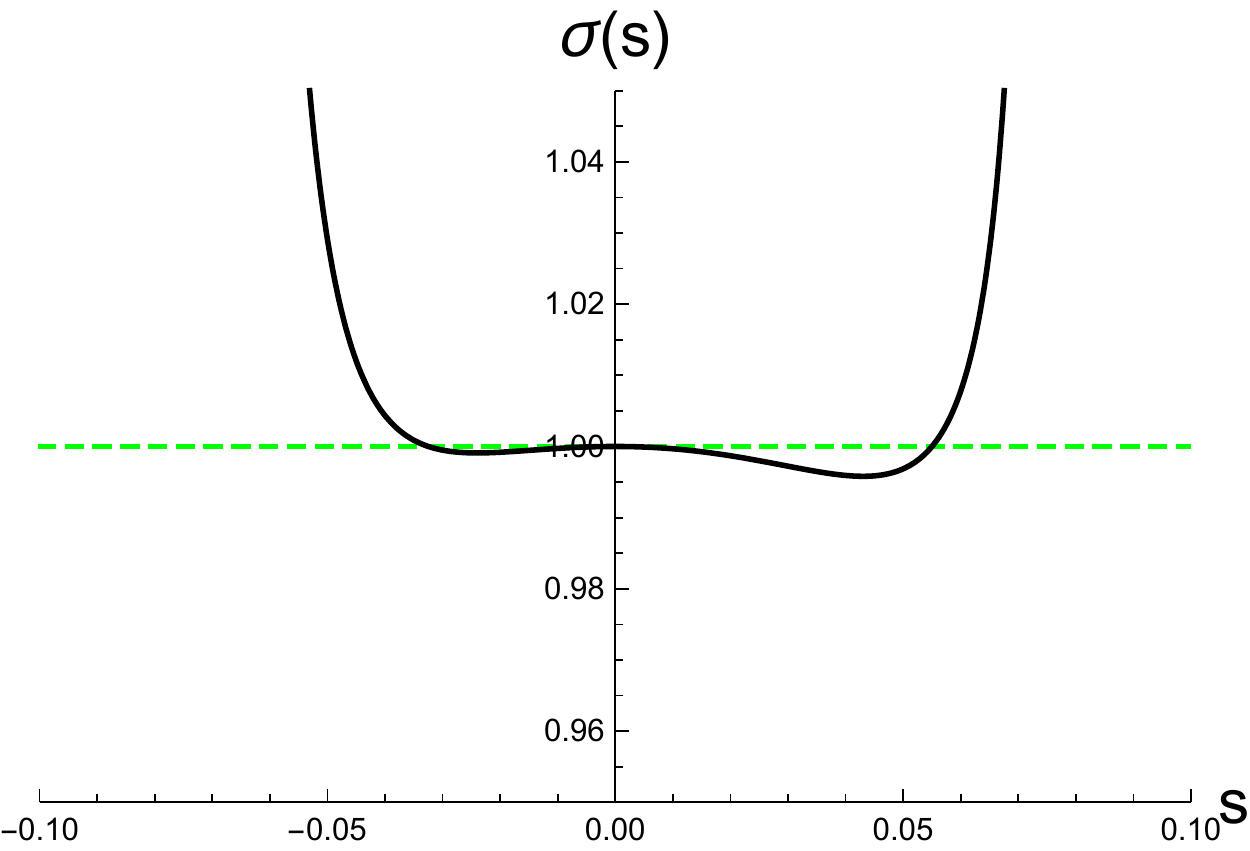}
  \caption{} \label{fig:4b}
\end{subfigure}
\begin{subfigure}{5cm}
  \centering
  \includegraphics[width=4.5cm]{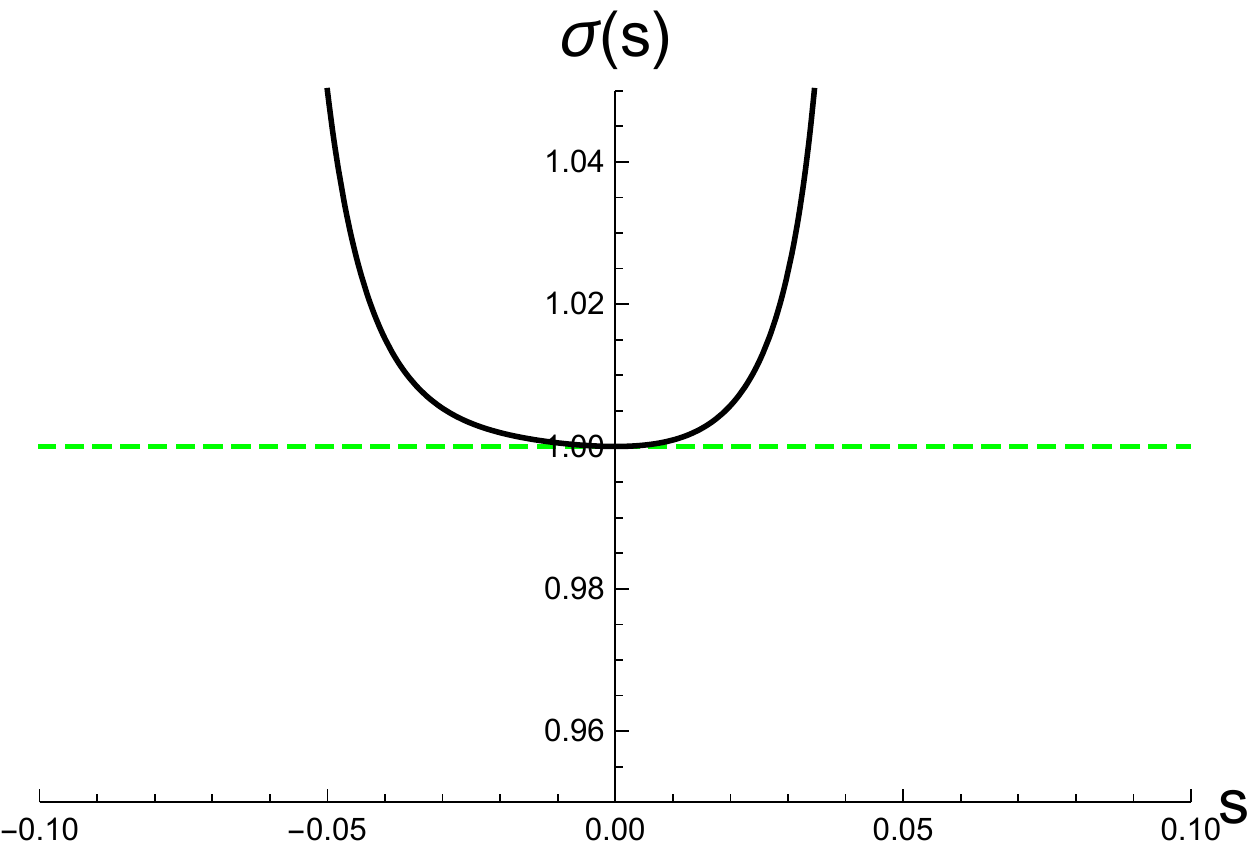}
  \caption{} \label{fig:4c}
\end{subfigure}
\caption{(A) A plot of $\sigma(s)$ for $\thetabt_0 = \thetab_0$. Note that $\sigma(s)$ is extra flat at $s = 0$. (B) A plot of $\sigma(s)$ for $\thetabt_0 = \thetab_0 + \epsilonb_1$. Note that $\sigma(s)$ develops a local maximum at $s = 0$. (C) $\sigma(s)$ for $\thetabt_0 = \thetab_0 + \epsilonb_2$. Note that $\sigma(s)$ develops a local minimum at $s = 0$. Again, the dashed green lines are there to help visualize the behavior of $\sigma(s)$ near $s = 0$. Note that all three plots have characteristics similar to the plots of certain quartic polynomials.} \label{fig:4}
\end{figure}

We note that the equations $a_k(\thetab_0,\bg) = 0$ become complex at an alarming rate as $n$ or $k$ increases, and it is for this reason only that we do not push this method farther to search for a higher order $S$-curve.

Next, now that we know the sets \eqref{eq:set1} and \eqref{eq:set2} are non-empty, it is natural to ask if they have any discernible structure or if we can estimate their size. Although we do not provide results in this direction, we demonstrate an interesting relationship in Figure \ref{fig:5}. It would be interesting to determine if the sets \eqref{eq:set1} and \eqref{eq:set2} are closely related to the order parameter in general. Finally, another class of sets that would be interesting to study are
\begin{align*}
\{ \gammab \in \R_{>0}^n : \det_\red(\Jb(\thetab,\gammab)) = 0 \text{ and } \delta(\thetab,\gammab) < 0 \}
\end{align*}
for fixed $\thetab \in \R^n$. In particular, estimating its measure would be an interesting problem.

\begin{figure}[H]
\captionsetup{font=scriptsize}
\centering
\begin{subfigure}{7cm}
  \centering
  \includegraphics[width=7cm]{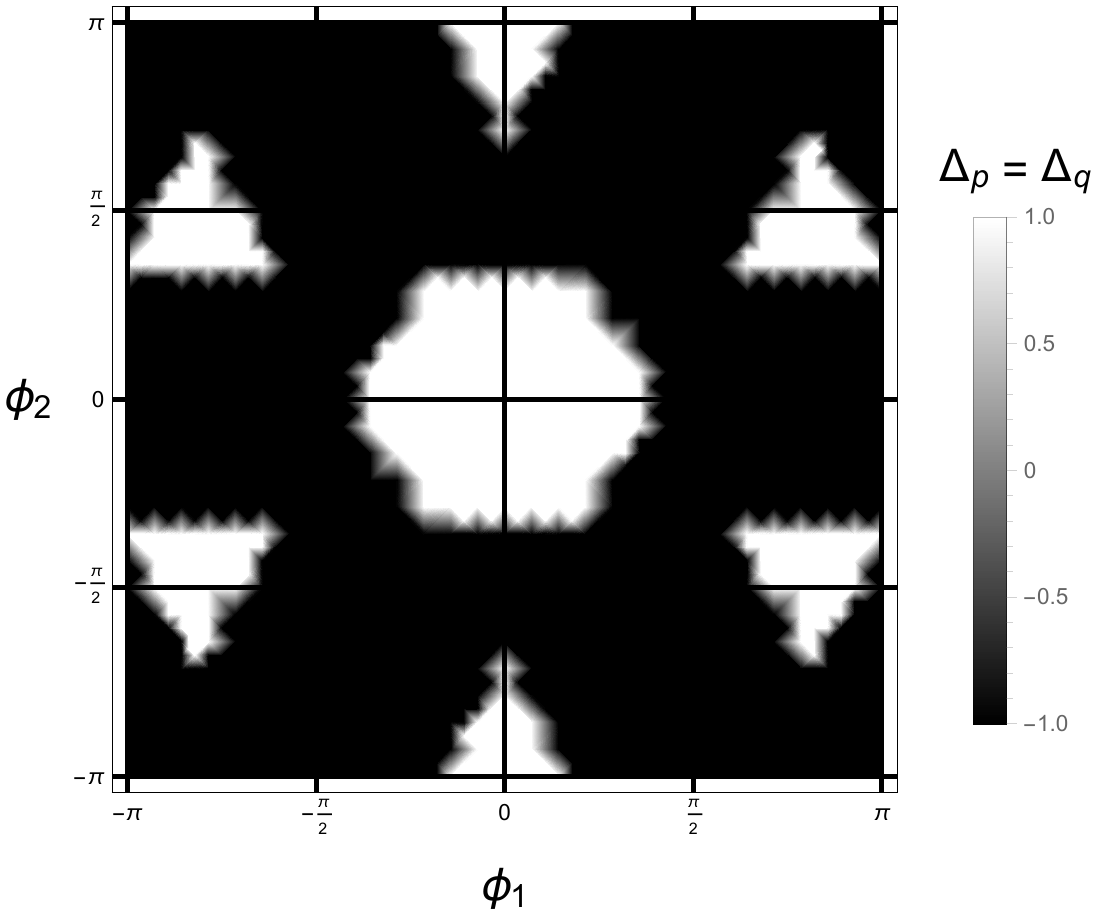}
  \caption{} \label{fig:5a}
\end{subfigure}
\begin{subfigure}{7cm}
  \centering
  \includegraphics[width=7cm]{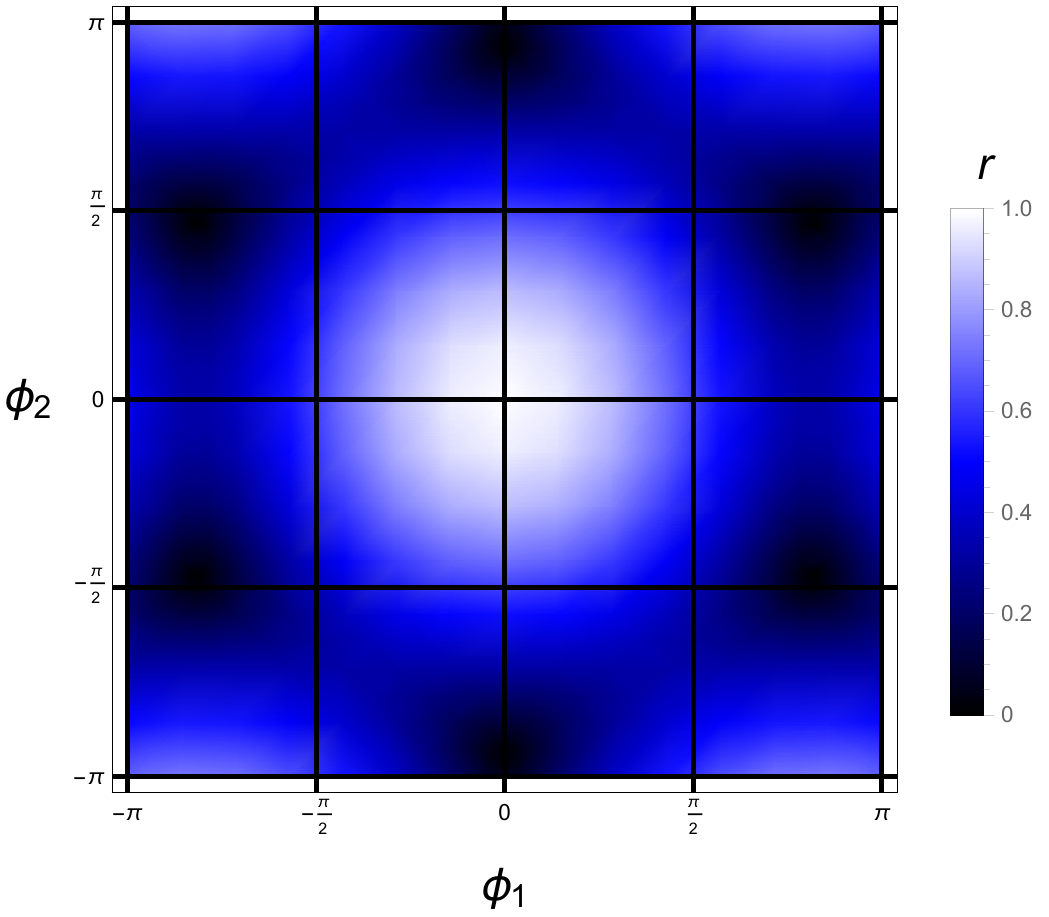}
  \caption{} \label{fig:5b}
\end{subfigure}
\caption{(A) A density plot of $\Delta_p(\thetab) = \Delta_q(\thetab)$. (Note that equality holds almost everywhere by our discussion following Theorem \ref{thm:criterion2}.) (B) A density plot of the order parameter $r$. In both plots we again projected $\R^3$ onto the mean-zero plane $\1^\perp$ via the variables $\phib = \Pb \thetab$. Note that the white regions in (A) correspond to the white and black regions in (B) and that the black region in (A) corresponds to the blue region in (B). Therefore the region $\{ \thetab \in \R^3 : \Delta_p(\thetab) = \Delta_q(\thetab) = -1\}$ is roughly the set of $\thetab$ for which $r$ isn't too large or too small.} \label{fig:5}
\end{figure}



\section{Conclusion}

We derived a criterion for the Kuramoto model on a ring network with positive coupling to exhibit a bifurcation in which two branches of phase-locked solutions collide as $\sigma$ increases. Furthermore, we derived a sufficient condition for one of these branches to consist of stable phase-locked solutions. We then applied our criteria to show that for any $n \ge 3$ we can choose $\omegab$ and $\bg$ so that this bifurcation occurs. (We require that $n \ge 5$ to guarantee the existence of $\omegab$ and $\bg$ for which one of the branches consists of stable phase-locked solutions.) We conjecture that generically our bifurcation is locally a subcritical bifurcation and globally an $S$-curve.



\section{Acknowledgements}

The author gratefully acknowledges support under NSF grant DMS1615418.



\section{Appendix}

We start by making a definition that significantly shortens the proofs of Lemma \ref{lem:quadratic} and Lemma \ref{lem:st}.

\begin{definition} \label{def:d}
Let $\ub \in \R^n$ and $f : \R \rightarrow \R$ be given. Define the diagonal matrix $\Db_\ub \in \R^{n \times n}$ and the vector $f(\ub) \in \R^n$ by
\begin{align*}
(\D_\ub)_{ii} = u_i \quad \text{and} \quad f(\ub)_i = f(u_i).
\end{align*}
\end{definition}

Now we give proofs of all relevant theorems and lemmas.

\begin{proof}[Proof of Theorem \ref{thm:criterion1}]
Note that $\nb(\theta_0,\gammab)$ is a normal vector for the level set $\{ \thetab \in \R^n : \det_\red(\Jb(\thetab,\gammab)) = 0\}$ at $\thetab_0$ pointing in the direction of most rapid increase for $\det_\red(\Jb(\thetab,\gammab))$. From \eqref{eq:hypothesis} we conclude that
\begin{align} \label{eq:detsigma}
\det_\red(\Jb(\thetab(s),\gammab))
\begin{cases}
> 0 & \mbox{if $s < 0$ is small,}
\\
= 0 & \mbox{if $s = 0$,}
\\
< 0 & \mbox{if $s > 0$ is small,}
\end{cases}
\quad \text{and} \quad
\frac{d \sigma}{ds}
\begin{cases}
> 0 & \mbox{if $s < 0$ is small,}
\\
= 0 & \mbox{if $s = 0$,}
\\
< 0 & \mbox{if $s > 0$ is small,}
\end{cases}
\end{align}
where $\thetab(s)$ and $\sigma(s)$ are given in Lemma \ref{lem:flow}. In other words, $\sigma(0) = \sigma_0$ is a local maximum.

Next, by Theorem 2.9 of \cite{doi:10.1137/15M1034258} we have that
\begin{align*}
\text{$\#$positive eigenvalues of $\Jb(\thetab,\gammab)$} = \# \{ i : \gamma_i \cos \eta_i > 0 \} -
\begin{cases}
1 & \mbox{ if $\sum_{i=1}^n \frac{1}{\gamma_i \cos \eta_i} > 0$,}
\\
0 & \mbox{ if $\sum_{i=1}^n \frac{1}{\gamma_i \cos \eta_i} < 0$.}
\end{cases}
\end{align*}
For all $\thetab$ in a sufficiently small neighborhood of $\thetab_0$ we have that $\# \{ i : \gamma_i \cos \eta_i > 0 \} = n_+(\cos \etab_0)$ and that
\begin{align*}
\sum_{i=1}^n \frac{1}{\gamma_i \cos \eta_i} = \frac{1}{n} \left( \prod_{i=1}^n \frac{1}{\gamma_i \cos \eta_i} \right) \det_\red(\Jb(\thetab,\gammab))
\end{align*}
has the same sign as $(-1)^{n-n_+(\cos \etab_0)} \det_\red(\Jb(\thetab,\gammab))$.
\end{proof}

\begin{proof}[Proof of Lemma \ref{lem:flow}]
We will only prove \eqref{eq:fixed}. To do this set $\rb(s) := \omegab - \sigma(s) \gb(\thetab(s),\gammab)$. Then
\begin{align*}
\Pb \frac{d \rb}{ds} &= -\sigma'(s) \Pb \gb(\thetab(s),\gammab) - \sigma(s) \Pb \Jb(\thetab(s),\gammab) \vb(\thetab(s),\gammab)
\\
&= -\sigma'(s) \Pb \gb(\thetab(s),\gammab) + \sigma(s) \Pb \Jb(\thetab(s),\gammab) \Pb^\top \adj(\Pb \Jb(\thetab(s),\gammab) \Pb^\top) \Pb \gb(\thetab(s),\gammab)
\\
&= [-\sigma'(s) + \sigma(s) \det_\red(\Jb(\thetab(s),\gammab)] \Pb \gb(\thetab(s),\gammab) = \0.
\end{align*}
But $\rb$ is orthogonal to $\1$ which spans the nullspace of $\Pb$. Therefore we conclude that $\rb(s) = \rb(0) = \0$.
\end{proof}

\begin{proof}[Proof of Lemma \ref{lem:quadratic}]

We first note that \eqref{eq:quadratic} follows from the identity
\begin{align} \label{eq:btv}
\Bb^\top \vb(\thetab,\gammab) = \det_\red(\Jb(\thetab,\gammab)) \left( \left( \frac{\1^\top \Db_{\sin \etab} \Db_{\cos \etab}^{-1} \1}{\1^\top \Db_\gammab^{-1} \Db_{\cos \etab}^{-1} \1} \right) \Db_\gammab^{-1} - \Db_{\sin \etab} \right) \Db_{\cos \etab}^{-1} \1 = \Cb(\thetab) \hb(\gammab)
\end{align}
since the restriction $\Bb : \1^\perp \rightarrow \1^\perp$ is invertible and $\nb(\thetab,\gammab)^\top = n \hb(\gammab)^\top \nabla_\thetab \hb(\cos \etab) = \hb(\gammab)^\top \Ab(\thetab)$. For reference later we note the identities
\begin{align} \label{eq:gJ}
\gb(\thetab,\gammab) = \Bb \Db_\gammab \sin \etab \quad \text{and} \quad \Jb(\thetab,\gammab) = \Bb \Db_\gammab \Db_{\cos \etab} \Bb^\top.
\end{align}
We start from the equation $(\Pb \Jb(\thetab,\gammab) \Pb^\top) \Pb \vb(\thetab,\gammab) = -\det_\red(\Jb(\thetab,\gammab)) \Pb \gb(\thetab,\gammab)$ which simplifies to $\Jb(\thetab,\gammab) \vb(\thetab,\gammab) \\ = -\det_\red(\Jb(\thetab,\gammab)) \gb(\thetab,\gammab)$ since $\Pb \Pb^\top = \Ib$ and $\Pb^\top \Pb = \Ib - \frac{1}{n} \1 \1^\top$. By \eqref{eq:gJ} we have that
\begin{align*}
\Bb \Db_\gammab \Db_{\cos \etab} \Bb^\top \vb(\thetab,\gammab) = -\det_\red(\Jb(\thetab,\gammab)) \Bb \Db_\gammab \Db_{\sin \etab} \1
\end{align*}
hence there exists a scalar $x$ such that
\begin{align*}
\Bb^\top \vb(\thetab,\gammab) = -\det_\red(\Jb(\thetab,\gammab)) \Db_{\sin \etab} \Db_{\cos \etab}^{-1} \1 + x \Db_\gammab^{-1} \Db_{\cos \etab}^{-1} \1
\end{align*}
since the null space of $\Bb$ is spanned by $\1$. In fact we find that
\begin{align*}
x = \det_\red(\Jb(\thetab,\gammab)) \frac{\1^\top \Db_{\sin \etab} \Db_{\cos \etab}^{-1} \1}{\1^\top \Db_\gammab^{-1} \Db_{\cos \etab}^{-1} \1}
\end{align*}
by noting that $\1^\top \B^\top = \0$. This verifies the first half of \eqref{eq:btv}. Next we note that
\begin{align*}
\det_{\red}(\Jb(\thetab,\gammab)) \Db_{\sin \etab} \Db_{\cos \etab}^{-1} \1 = n \Db_{\sin \etab} \Db_{\cos \etab}^{-1} \1 \hb(\cos \etab)^\top \hb(\gammab)
\end{align*}
and that
\begin{align*}
\det_\red(\Jb(\thetab,\gammab)) \left( \frac{\1^\top \Db_{\sin \etab} \Db_{\cos \etab}^{-1} \1}{\1^\top \Db_\gammab^{-1} \Db_{\cos \etab}^{-1} \1} \right) \Db_\gammab^{-1} \Db_{\cos \etab}^{-1} \1 = n\left( \sum_{k=1}^n \frac{\sin \eta_k}{\cos \eta_k} \right) \Db_{\hb(\cos \etab)} \hb(\gammab)
\end{align*}
by Theorem \ref{thm:detred} where
\begin{align*}
n (\Db_{\sin \etab} \Db_{\cos \etab}^{-1} \1 \hb(\cos \etab)^\top)_{ij} = n \frac{\sin \eta_i}{\cos \eta_i} \hb(\cos \etab)_j
\end{align*}
and
\begin{align*}
n \left( \sum_{k=1}^n \frac{\sin \eta_k}{\cos \eta_k} \right) (\Db_{\hb(\cos \etab)})_{ij} =
\begin{cases}
n \left( \sum_{k=1}^n \frac{\sin \eta_k}{\cos \eta_k} \right) \hb(\cos \etab)_j & \mbox{if $i = j$,}
\\
0 & \mbox{if $i \ne j$,}
\end{cases}
\end{align*}
which verifies the second half of \eqref{eq:btv}.
\end{proof}

\begin{proof}[Proof of Lemma \ref{lem:min}]
Following the arguments in \cite{KREPS1984105} we easily find that
\begin{align*}
\min_{\xb \in \R_{\ge 0}^n} \frac{\xb^\top \Rb \xb}{\| \xb \|_1^2} \ge t \text{ if and only if } \min_{i \in \Ic} \det(\Rb_{\Ic,i}) \le 0 \text{ or } \frac{\det(\Rb_\Ic)}{\sum_{i \in \Ic} \det(\Rb_{\Ic,i})} \ge t
\end{align*}
for every $\Ic \subseteq \{1, \dots, n\}$. This is of course equivalent to
\begin{align*}
\min_{\xb \in \R_{\ge 0}^n} \frac{\xb^\top \Rb \xb}{\| \x \|_1^2} < t \text{ if and only if } \min_{i \in \Ic} \det(\Rb_{\Ic,i}) > 0 \text{ and } \frac{\det(\Rb_\Ic)}{\sum_{i \in \Ic} \det(\Rb_{\Ic,i})} < t
\end{align*}
for some $\Ic \subseteq \{1, \dots, n\}$. Taking the infimum over all such $t$ gives the result.
\end{proof}

\begin{proof}[Proof of Theorem \ref{thm:criterion2}]
Suppose that $\Delta_p(\thetab_0) = -1$. Then there exists an $\Ic \in \Ic(\thetab_0)$ such that
\begin{align*}
\lim_{\tau \rightarrow \infty} \frac{\det(\Rb_\tau(\thetab_0)_\Ic)}{\sum_{i \in \Ic} \det(\Rb_\tau(\thetab_0)_{\Ic,i})} = \epsilon < 0.
\end{align*}
By Lemma \ref{lem:min} there exists vectors $\xb_\tau \in \R_{\ge 0}^n$ such that $\| \xb_\tau \|_1 = 1$ and
\begin{align*}
\xb_\tau^\top \Rb_\tau(\thetab_0) \xb_\tau = (\hb(\cos \etab_0)^\top \xb_\tau)^2 \tau + \xb_\tau^\top \Tb(\thetab_0) \xb_\tau < \frac{\epsilon}{2}
\end{align*}
for all sufficiently large $\tau$. But the unit ball $\{ \xb \in \R^n : \| \xb \|_1 = 1\}$ is compact so there exists an infinite sequence $\tau_1 < \tau_2 < \dots$ with $\tau_n \rightarrow \infty$ such that $\lim_{n \rightarrow \infty} \xb_{\tau_n} = \xb$ exists. Now the second term $\xb_\tau^\top \Tb(\thetab_0) \xb_\tau$ is bounded since $\| \xb_\tau \|_1 = 1$. Therefore $\hb(\cos \etab_0)^\top \xb = \lim_{n \rightarrow \infty} \hb(\cos \etab_0)^\top \xb_{\tau_n} = 0$. Also, the first term $(\hb(\cos \etab_0)^\top \xb_\tau)^2$ is non-negative so that $\xb^\top \Tb(\thetab_0) \xb = \lim_{n \rightarrow \infty} \xb_{\tau_n}^\top \Tb(\thetab_0) \xb_{\tau_n} \le \epsilon/2$.

Clearly $\xb \in \R_{\ge 0}^n$ and $\| \xb \|_1 = 1$. Therefore $\hb(\cos \etab_0)^\top \xb = 0$ implies that $\hb(\cos \etab_0)$ has at least one positive component and at least one negative component. Therefore we can slightly modify $\xb$ to obtain a new vector $\xbt \in \R_{>0}^n$ such that $\hb(\cos \etab_0)^\top \xbt = 0$ and $\xbt^\top \Tb(\thetab_0) \xbt < 0$. By Lemma \ref{lem:quadratic} we can choose $\gammab \in \R_{>0}^n$ such that $\xbt = \hb(\gammab)$. Then $\det_\red(\Jb(\thetab_0,\gammab)) = \hb(\cos \etab_0)^\top \xbt = 0$ and $\delta(\thetab_0,\gammab) = \xbt^\top \Tb(\thetab_0) \xbt < 0$ so that \eqref{eq:hypothesis} is satisfied.

Now suppose that there exists a $\gammab$ satisfying \eqref{eq:hypothesis}. Then the vector $\xb = \hb(\gammab) \in \R_{>0}^n$ satisfies
\begin{align*}
\xb^\top \Rb_\tau(\thetab_0) \xb = \det_\red(\Jb(\thetab_0,\gammab))^2 \tau + \delta(\thetab_0,\gammab) < 0
\end{align*}
for all $\tau$. By Lemma \ref{lem:min} this implies that for each $\tau$ there exists an $\Ic_\tau \in \Ic(\thetab_0)$ such that $\det(\Rb_\tau(\thetab_0)_{\Ic_\tau}) < 0$. But $\Ic(\thetab_0)$ is finite so there exits an $\Ic \in \Ic(\thetab_0)$ such that $\det(\Rb_{\tau_n}(\thetab_0)_\Ic) < 0$ for a sequence $\tau_1 < \tau_2 < \dots $ with $\tau_n \rightarrow \infty$. However, $\det(\Rb_{\tau_n}(\thetab_0)_\Ic)$ is a polynomial in $\tau$ so that in fact $\det(\Rb_{\tau_n}(\thetab_0)_\Ic) < 0$ for all sufficiently large $\tau$. Therefore we obtain that $\Delta_q(\thetab_0) = -1$.
\end{proof}

\begin{proof}[Proof of Lemma \ref{lem:st}]
We first note that $\etabt = (\eta_1,\dots,\eta_n,0,\dots,0) \in \R^{\nt}$ where $\etab = (\eta_1,\dots,\eta_n) \in \R^n$. Therefore
\begin{align*}
\hb(\cos \etabt)_i =
\begin{cases}
\hb(\cos \etab)_i & \mbox{if $1 \le i \le n$,}
\\
\prod_{j=1}^n \cos \eta_j & \mbox{if $n < i \le \nt$,}
\end{cases}
\end{align*}
demonstrating the first equality. Next we note that
\begin{align*}
\Cbt(\thetabt) =
\begin{pmatrix}
\frac{\nt}{n} \Cb(\thetab) & \hdots
\\
\0 & \ddots
\end{pmatrix}
\end{align*}
from Definition \ref{def:abc}. Also we have that
\begin{align*}
\Hbt(\cos \etabt) =
\begin{pmatrix}
\Hb(\cos \etab) & \hdots
\\
\vdots & \ddots
\end{pmatrix}
\quad \text{and} \quad \Abt(\thetabt) = \nt \Hbt(\cos \etabt) \Dbt_{\sin \etabt} \Bbt^\top \quad \text{and} \quad \Ab(\thetab) = n \Hb(\cos \etab) \Db_{\sin \etab} \Bb^\top
\end{align*}
where $\Hbt$ and $\Hb$ are the Jacobians of $\hbt$ and $\hb$ respectively. Therefore
\begin{align*}
\Tbt(\thetabt) &= \nt \Hbt(\cos \etabt) \Dbt_{\sin \etabt} \Bbt^\top (\Bbt^\top)^{-1} \Cbt(\thetabt) =
\begin{pmatrix}
\frac{\nt}{n} \Hb(\cos \etab) & \hdots
\\
\vdots & \ddots
\end{pmatrix}
\begin{pmatrix}
\Db_{\sin \etab} & \0
\\
\0 & \0
\end{pmatrix}
\begin{pmatrix}
\frac{\nt}{n} \Cb(\thetab) & \hdots
\\
\0 & \ddots
\end{pmatrix}
\\
&=
\begin{pmatrix}
(\frac{\nt}{n})^2\Hb(\cos \etab) \Db_{\sin \etab} \Cb(\thetab) & \hdots
\\
\vdots & \ddots
\end{pmatrix}
=
\begin{pmatrix}
(\frac{\nt}{n})^2\Hb(\cos \etab) \Db_{\sin \etab} \Bb (\Bb^\top)^{-1} \Cb(\thetab) & \hdots
\\
\vdots & \ddots
\end{pmatrix}
=
\begin{pmatrix}
(\frac{\nt}{n})^2\Tb(\thetab) & \hdots
\\
\vdots & \ddots
\end{pmatrix}
\end{align*}
completing the proof.
\end{proof}



\bibliography{StabilityKuramotoFrequencies}
\bibliographystyle{plain}

\end{document}